\documentclass[reqno,11pt]{amsart}

\usepackage{xcolor}
\usepackage{graphicx}
\usepackage[hidelinks]{hyperref}
\usepackage{amscd}

\usepackage{amsmath}
\usepackage{amsthm}
\usepackage{amssymb}
\usepackage{latexsym,array}
\usepackage{amsfonts}
\usepackage{shadow}
\usepackage{amsbsy}
\usepackage{amssymb}
\usepackage{dsfont}
\usepackage{mathtools}
\usepackage{enumitem}
\usepackage{doi}

\usepackage[notref, notcite, final]{showkeys}
\usepackage{color}
\usepackage{graphicx}
\usepackage[hidelinks]{hyperref}
\usepackage{cleveref}
\usepackage{fullpage}
\usepackage{comment}
\usepackage{tikz-cd}
\usepackage{tikz}
\usepackage{float}

\usepackage{dsfont}

\usepackage{cleveref}

\numberwithin{equation}{section}

\newcommand{\bigD}{\mathcal{D}}

\newcommand{\PRes}{\mathcal{Q}}

\newtheorem{theorem}{Theorem}[section]
\newtheorem{lemma}[theorem]{Lemma}

\newtheorem{proposition}[theorem]{Proposition}
\newtheorem{prob}[theorem]{Problem}

\theoremstyle{definition}
\newtheorem{remark}[theorem]{{\bf Remark}}
\newtheorem{definition}[theorem]{Definition}

\newcommand{\cc}{\mathbb{C}}

\newcommand{\hh}{\mathbb{H}}

\newcommand{\pp}{\partial}

\renewcommand{\Re}{\mathrm{Re}}

\crefname{enumi}{}{}
\crefname{enumii}{}{}

\title[Properties of a polyanalytic functional \\ calculus on the $S$-spectrum]{Properties of a polyanalytic functional \\ calculus on the $S$-spectrum}
\author{Antonino De Martino, Stefano Pinton}
\date{}

\begin{document}
	
	\maketitle
	
\begin{abstract}
The Fueter mapping theorem gives a constructive way to extend holomorphic functions of one complex variable to monogenic functions, i.e., null solutions of the generalized Cauchy-Riemann operator in $\mathbb{R}^4$, denoted by $\mathcal{D}$. This theorem is divided in two steps. In the first step a holomorphic function is extended to a slice hyperholomorphic function. The Cauchy formula for these type of functions is the starting point of the $S$-functional calculus.
In the second step a monogenic function is obtained by applying the Laplace operator in four real variables, namely $ \Delta$, to a slice hyperholomorphic function. The polyanalytic functional calculus, that we study in this paper, is based on the factorization of $\Delta= \mathcal{D} \mathcal{\overline{D}}$. Instead of applying directly the Laplace operator to a slice hyperholomorphic function we apply first the operator $ \mathcal{\overline{D}}$ and we get a polyanalytic function of order 2, i..e, a function that belongs to the kernel of $ \mathcal{D}^2$. We can represent this type of functions in an integral form and then we can define the polyanalytic functional calculus on $S$-spectrum. The main goal of this paper is to show the principal properties of this functional calculus. In particular, we study a resolvent equation suitable for proving a product rule and generate the Riesz projectors.

\end{abstract}

\medskip
\noindent AMS Classification  47A10, 47A60.


\section{Introduction}
The polyanalytic functional calculus of order 2 on the $S$-spectrum was introduced in \cite{DP}. Similarly as the harmonic functional calculus (see \cite{CDPS}), it can be seen as an intermediate functional calculus between the $S$-functional calculus and the $F$-functional calculus (see \cite{CGKBOOK}).

To better understand the nature of these functional calculi we recall that the spaces of functions over which they are defined, are strictly related to the Fueter-Sce mapping theorem. This theorem is based on a two steps procedure: the first step extends holomorphic functions of one complex variable to slice hyperholomorphic functions, the second one gives a monogenic function by applying the Laplace operator in four real variables to a slice hyperholomorphic function (see \cite{F}). The Fueter construction can be visualized as follows 
\begin{equation}
\label{diag1}
\begin{CD}
	\textcolor{black}{\mathcal{O}(D)}  @>T_{F1}>> \textcolor{black}{SH(\Omega_D)}  @>\ \   T_{F2}=\Delta >>\textcolor{black}{AM(\Omega_D)},
\end{CD}
\end{equation}

where $\mathcal O(D)$ is the set of holomorphic functions on $D$, $SH(\Omega_D)$ is the set of slice hyperholomorphic functions induced on $\Omega_D$, $AM(\Omega_D)$ is the set of axially monogenic functions on $\Omega_D$, $T_{F_1}$ denotes the first linear operator of the Fueter construction, called slice operator. Associated to the spaces of functions $SH(\Omega_D)$ and $AM(\Omega_D)$ there are two functional calculi. The first one is the $S$-functional calculus and it is based on the Cauchy integral formula for slice hyperholomorphic functions (see Theorem \ref{Cauchy}). By applying the Laplace operator to the slice hyperholomorphic Cauchy formula we obtain the integral representation of an axially monogenic function. This is the so-called Fueter theorem in integral form (see Theorem \ref{Fueter}). The  $F$-functional calculus  is based on this theorem.
The following diagram illustrates these relations
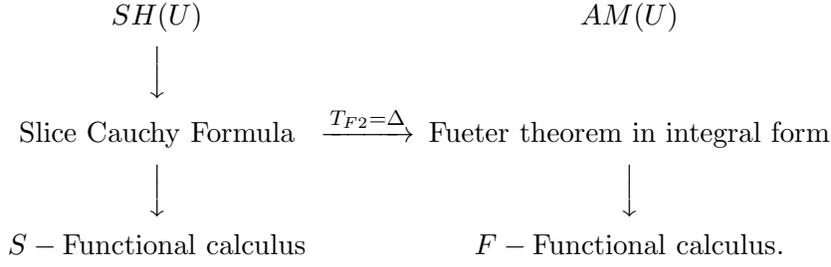
\begin{equation*}
	\begin{CD}
		{SH(U)} @.  {AM(U)} \\   @V  VV
		@.
		\\
		{{\rm  Slice\ Cauchy \ Formula}}  @> T_{F2}=\Delta>> {{\rm Fueter\ theorem \ in \  integral\  form}}
		\\
		@V VV    @V VV
		\\
		{S-{\rm Functional \ calculus}} @. {F-{\rm Functional \ calculus}}.
	\end{CD}
\end{equation*}
To proceed further we fix the notations. We define the quaternions as follows

$$
\mathbb{H}=\lbrace{q=q_0+q_1e_1+q_2e_2+q_3e_3 \ \ |\   \ q_0,q_1,q_2,q_3\in\mathbb{R}}\rbrace,
$$
where the imaginary units satisfy the relations
$$e_1^2=e_2^2=e_3^2=-1\quad \text{and}\quad e_1e_2=-e_2e_1=e_3,\ \  e_2e_3=-e_3e_2=e_1, \ \  e_3e_1=-e_1e_3=e_2.$$
Given $q\in\mathbb H$
we call $ \Re(q):=q_0$ the real part of $q$ and
$ \underline{q}\,=q_1e_1+q_2e_2+q_3e_3$  the imaginary part.
The modulus of $q\in \mathbb{H}$ is given by
$|q|=\sqrt{q_0^2+q_1^2+q_2^2+q_3^2},$ the conjugate of $q$ is defined by
$\overline{q}=q_0- \underline{q}$ and we have $|q|=\sqrt{q\overline{q}}$.
\\We recall that the Fueter operator $\mathcal{D}$ and its conjugate $\overline{\mathcal{D}}$ are defined as follows
$$
\mathcal{D}:= \partial_{q_0}+ \sum_{i=1}^{3} e_i \partial_{q_i}
\ \ \ \
{\rm and}\ \ \ \
\overline{\mathcal{D}}:= \partial_{q_0}- \sum_{i=1}^{3} e_i \partial_{q_i}.
$$

The Laplace operator $\Delta$ can be factorized by the Fueter operator $\mathcal D$ and its conjugate $\overline{\mathcal D}$: $T_{F_2}=\Delta=\mathcal D\overline{\mathcal D}=\overline{\mathcal D}\mathcal D$. As a consequence of the Fueter mapping theorem when we apply $\mathcal D$ to a slice hyperholomorphic function we obtain an axially harmonic function, whereas, when we apply $\overline{\mathcal D}$ to a slice hyperholomorphic function we obtain an axially polyanalytic function of order 2. Thus we have the following diagrams  
\begin{equation}
\label{diagg}
	\begin{CD}
		\textcolor{black}{\mathcal{O}(D)}  @>T_{F1}>> \textcolor{black}{SH(\Omega_D)}  @>\ \   \mathcal{D}>>\textcolor{black}{AH(\Omega_D)}
		@>\ \   \overline{\mathcal{D}} >>\textcolor{black}{AM(\Omega_D)},
	\end{CD}
\end{equation}

\begin{equation}
\label{diag3}
	\begin{CD}
		\textcolor{black}{\mathcal{O}(D)}  @>T_{F1}>> \textcolor{black}{SH(\Omega_D)}  @>\ \   \overline{\mathcal{D}}>>\textcolor{black}{AP_2(\Omega_D)}
		@>\ \   \mathcal{D} >>\textcolor{black}{AM(\Omega_D)}.
	\end{CD}
\end{equation}
where $AH(\Omega_D)$ is the set of axially harmonic functions and $AP_2(\Omega_D)$ is the set of axially polyanalytic functions of order 2. Basically, the diagrams \eqref{diagg} and \eqref{diag3} are factorizations of the diagram \eqref{diag1}. These diagrams lead to the definition of \emph{fine structure}.
\begin{definition}
We will call {\em fine structure of the  spectral theory on the $S$-spectrum}
the set of functions spaces and the associated functional calculi
induced by a factorization of the operator $T_{F2}$, in the Fueter extension theorem.
\end{definition}

In the Clifford algebra setting the diagrams \eqref{diagg} and \eqref{diag3} are much more involved, see \cite{FIVEDIM}. This is due to fact that the map $T_{F2}$ becomes the so-called Fueter-Sce map $T_{FS2}= \Delta^{\frac{n-1}{2}}_{n+1}$, where $n$ is an odd number, see \cite{CSSS}.

\medskip

In literature the theory of polyanalytic functions plays an important role for studying some elasticity problems, see \cite{K,M}, and investigating some aspects of the time frequency analysis, see for instance \cite{A,A1, DMD2}. Recently the polyanalytic functions has been studied in the slice hyperholomorphic setting, see \cite{polyan1, polyan2}, and also a Fueter theorem has been considered in \cite{ADS}. Moreover a slice polyanalytic functional calculus has been considered in \cite{ACDS2}. For further information about the polyanalytic function theory see \cite{B1}, for more applications about this theory see \cite{AF}. 

\medskip

As for the axially monogenic functions, it is possible to obtain an integral representation of the axially harmonic functions and axially polyanalytic functions of order 2. Based on these integral representations we define the corresponding functional calculi: the harmonic functional calculus ($Q$-functional calculus) and the polyanalytic functional calculus  of order 2 ($P_2$-functional calculus). It is possible to visualize these relations with the following diagrams 
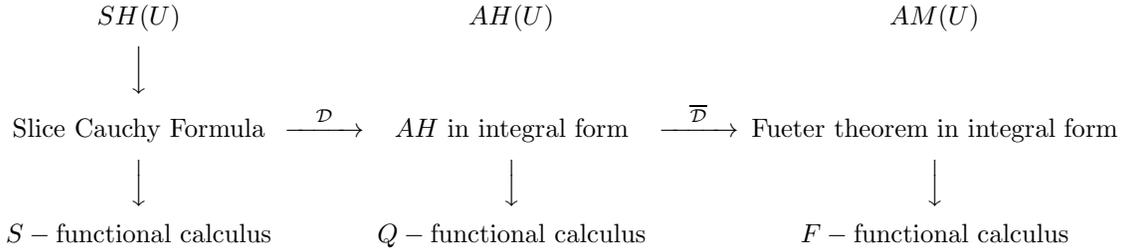
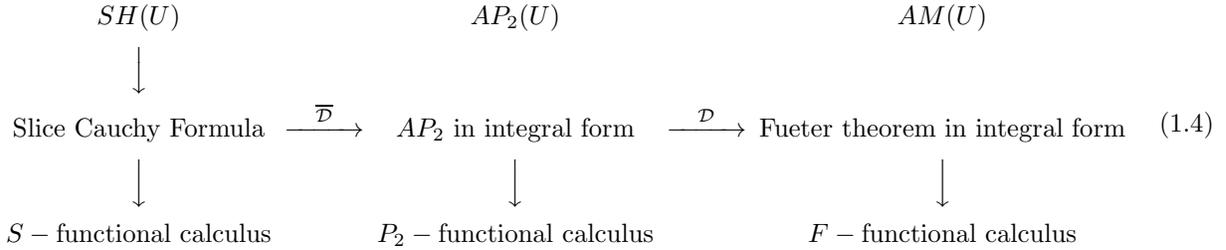
\newline
\newline
{\small
	\begin{equation*}
		\begin{CD}
			{SH(U)} @. {AH(U)}  @.  {AM(U)} \\   @V  VV
			@.
			\\
			{{\rm  Slice\ Cauchy \ Formula}}  @> \mathcal{D} >> {AH {\rm \ in \  integral\  form}}@> \overline{\mathcal{D}} >> {{\rm Fueter\ theorem \ in \  integral\  form}}
			\\
			@V VV    @V VV  @V VV
			\\
			S-{{\rm functional \ calculus}} @. {Q-{\rm functional \ calculus}}@. F-{{\rm functional \ calculus}}
		\end{CD}
	\end{equation*}
}
\newline
\newline
and
\newline
\newline
{\small
	\begin{equation}
		\begin{CD}
			{SH(U)} @. {AP_2(U)}  @.  {AM(U)} \\   @V  VV
			@.
			\\
			{{\rm  Slice\ Cauchy \ Formula}}  @>\overline{ \mathcal{D}} >> {AP_2 {\rm \ in \  integral\  form}}@> \mathcal{D} >> {{\rm Fueter\ theorem \ in \  integral\  form}}
			\\
			@V VV    @V VV  @V VV
			\\
			S-{{\rm functional \ calculus}} @. P_2-{{\rm functional \ calculus}}@. F-{{\rm functional \ calculus}}
		\end{CD}
	\end{equation}
}
\newline
\newline
Once we have proved the good definitions of these functional calculi, a natural field of investigation is the determination of their main properties which are: the algebraic properties, the resolvent equation, the Riesz projectors, the product rule. For further information about the $S$-functional calculus see \cite{CGKBOOK, CSS3}, whereas for the $F$-functional calculus see \cite{CDS, CGKBOOK, CS}.
The properties of the $Q$-functional calculus are studied in \cite{CDPS}. 

\medskip

The goal of this paper is to investigate the properties of the $P_2$-functional calculus.

\medskip

\emph{Outline of the paper:} The paper consists of six sections, the first one being this introduction. In Section 2 we give some basic notions of the $S$-functional calculus, the $F$-functional calculus, the $Q$-functional calculus and  the underlying theory of the slice hyperholomoprhic functions. In Section 3 we recall the definition of the $P_2$-functional calculus and we prove some algebraic properties. In Section 4 we show a resolvent equation for the $P_2$-functional calculus, by this fundamental tool, we prove the product rule. In this section it is also proved a product rule for the $F$-functional calculus based on the $P_2$-functional calculus and the $Q$-functional calculus. In Section 5 we study the Riesz projectors for the $P_2$-functional calculus. Finally, in Section 6 we prove a different version of the product rule for the $Q$-functional calculus, based on a new resolvent equation.  
\section{Preliminaries results on functions and operators}

In this section we recall some basic notions about the slice hyperholomorphic functions. Moreover, we give the definitions and the properties of all the functional calculi that we need: $S$-functional calculus, $F$-functional calculus and $Q$-functional calculus. 	
\subsection{Hyperholomorphic functions}
Before to introduce the definition of slice hyperholomorphic function, we need some preliminary notations.

Let us denote by $ \mathbb{S}$ the unit sphere of purely imaginary quaternions
$$ \mathbb{S}= \{\underline{q}=q_1e_1+q_2e_2+q_3e_3\ \ | \ \  \, q_1^2+q_2^2+q_3^2=1\} .$$
Notice that if $J \in \mathbb{S}$, then $J^2=-1$. Therefore $J$ is an imaginary unit, and we denote by
$$ \mathbb{C}_J=\{u+Jv\ \ \ | \ \ u,v \in \mathbb{R}\},$$
an isomorphic copy of the complex numbers. Given a non-real quaternion $q= q_0+ \underline{q}= q_0+J_q | \underline{q}|$, we set $J_q= \underline{q}/ | \underline{q}|  \in \mathbb{S}$ and we associate to $q$ the 2-sphere defined by
$$ [q]:= \{q_0+J |\underline{q}| \ \ | \ \ J \in \mathbb{S}\}.$$
\begin{definition}
	Let $U \subseteq \mathbb{H}$.
	\begin{itemize}
		\item We say that $U$ is axially symmetric if, for every $u+Iv \in U$, all the elements $u+Jv$ for $J \in \mathbb{S}$ are contained in $U$.
		\item We say that $U$ is a slice domain if $U \cap \mathbb{R} \neq \emptyset$ and if $U \cap \mathbb{C}_J$ is a domain in $\mathbb{C}_J$ for every $J \in \mathbb{S}$.
	\end{itemize}
	
\end{definition}
\begin{definition}
	An axially symmetric open set $U \subset \mathbb{H}$ is called slice Cauchy domain if $U \cap \mathbb{C}_J$ is a Cauchy domain in $ \mathbb{C}_J$ for every $J \in \mathbb{S}$. More precisely, $U$ is a slice Cauchy domain if, for every $J \in \mathbb{S}$, the boundary of $ U \cap \mathbb{C}_J$ is the union of a finite number of nonintersecting piecewise continuously differentiable Jordan curves in $ \mathbb{C}_J$.
\end{definition}
The axially symmetric open sets are the suitable domains for the slice hyperholomorphic functions.
\begin{definition}[Slice hyperholomorphic functions]
	\label{hyper}
	Let $U \subseteq \mathbb{H}$ be an axially symmetric open set and let
	$$
	\mathcal{U} =\{ (u,v)\in \mathbb{R}^2 \ | \ u+\mathbb{S}v \in U\}.
	$$
	We say that a function $f: U \to \mathbb{H}$ of the form
	$$ f(q)=\alpha(u,v)+J\beta(u,v)$$
	is left slice hyperholomorphic if $\alpha$ and $\beta$ are $ \mathbb{H}$-valued differentiable functions such that
	\begin{equation}\label{eveodd}
		\alpha(u,v)=\alpha(u,-v), \ \ \ \  \beta(u,v)=-\beta(u,-v) \ \\ \ \  \hbox{for all} \, \, (u,v) \in \mathcal{U},
	\end{equation}
	and if $\alpha$ and $\beta$ satisfy the Cauchy-Riemann system
	$$ \partial_u \alpha(u,v)- \partial_v \beta(u,v)=0, \quad \partial_v \alpha(u,v)+ \partial_u \beta(u,v)=0.$$
A function right slice hyperholomorphic, if it is of the form
	$$ f(q)= \alpha(u,v)+\beta(u,v)J,$$
	where $\alpha$, $\beta$ satisfy the above conditions.
\end{definition}
The set of left (resp. right) slice hyperholomorphic functions on $U$ is denoted by the symbol $SH_{L}(U)$ (resp. $SH_{R}(U)$). The subset of intrinsic slice hyperholomorphic functions consists of those slice hyperholomorphic functions such that $\alpha$, $\beta$ are real-valued function and is denoted by $ N(U)$.

\medskip

Another class of hyperholomorphic functions that appears in the Fueter construction is the following
\begin{definition}[Fueter regular functions]
Let $U\subset \mathbb H$ be an open set. A real differentiable function $ f: U \to \mathbb{H}$ is called (left) Fueter regular if
$$\mathcal{D} f(q):=\partial_{q_0} f(q)+\sum_{i=1}^3 e_i \partial_{q_i}f(q)=0.$$
\end{definition}
A bridge between the slice hyperholomorphic and the monogenic functions is the Fueter theorem, see \cite{F, CSSS}.

\begin{theorem}[Fueter mapping theorem]
	\label{F1b}
	Let $f_{0}(z)= \alpha(u,v)+i \beta(u,v)$ be a holomorphic function defined in a domain (open and connected) $D$ in the upper-half complex plane and let
	$$ \Omega_D=\{q=q_0+\underline{q} \,\  |\  \ (q_0, |\underline{q}|) \in D\}$$
	be the open set induced by $D$ in $\mathbb{H}$. Then the operator $T_{F1}$ defined by
\begin{equation}
\label{slice}
f(q)= T_{F1}(f_0):= \alpha(q_0, |\underline{q}|)+ \frac{\underline{q}}{|\underline{q}|}\beta(q_0, |\underline{q}|)
\end{equation}
	maps the set of holomorphic functions in the set of intrinsic slice hyperholomorphic functions. Moreover, the function
	$$ \breve{f}(q):=T_{F2} \left(\alpha(q_0, |\underline{q}|)+ \frac{\underline{q}}{|\underline{q}|}\beta(q_0, |\underline{q}|)\right),$$
	where $T_{F2}= \Delta$ and $\Delta$ is the Laplacian in four real variables $q_{\ell}$, $ \ell=0,1,2,3$, is in the kernel of the Fueter operator i.e.
	$$ \mathcal{D} \breve{f}=0 \quad \hbox{on} \quad \Omega_D.$$
\end{theorem}

We will consider also  polyanalytic Fueter regular functions, see \cite{B1976}.

\begin{definition}
Let $U \subset \mathbb{H}$ be an open set and let $f: U \to \mathbb{H}$ be a function of class $ \mathcal{C}^n$, with $n \geq 1$. We say that $f$ is (left) polyanalytic Fueter  regular of order $n$ on $U$ if
$$ \mathcal{D}^{n}f(q)= \left(\partial_{q_0} +\sum_{i=1}^3 e_i \partial_{q_i}
\right)^n f(q)=0.$$
\end{definition}
For this kind of functions it is possible to give the following characterization, see \cite{B1976, DB1978}.
\begin{proposition}
A function $f$ is polyanalytic Fueter regular of order $n$ if and only if it can be decomposed in terms of unique Fueter regular functions $ \phi_0(q)$,..., $ \phi_{n-1}(q)$ such that we have
$$ f(q)= \sum_{k=0}^{n-1} x_0^k \phi_k(q).$$
\end{proposition}

Now, let us consider $D$ be a  domain in the upper-half complex plane. Let $ \Omega_D$ be an axially symmetric open set in $ \mathbb{R}^4$ and let $x=x_0+ \underline{x}=x_0+ r \underline{\omega} \in \Omega_D$. We say that a function $f: \Omega_D \to \mathbb{H}$ is of axial type if there exist two quaternionic-valued functions $A(q_0,r)$ and $B(q_0,r)$ independent of $\underline{\omega} \in \mathbb{S}$ such that
$$ f(q)=A(q_0,r)+ \underline{\omega} B(q_0,r), \qquad r>0.$$ 

\begin{definition}[axially monogenic function]
Let $f: \Omega_D \subset \mathbb{R}^4 \to \mathbb{H}$ be of axial type and of class $ \mathcal{C}^3(\Omega_D)$. Then the function
$$ \breve{f}(q):= \Delta f(q) \qquad \hbox{on} \quad \Omega_D$$
is called axially monogenic, since by the Fueter theorem mapping theorem, it satisfies 
$$ \mathcal{D} \breve{f}(q)=0\qquad \hbox{on} \quad \Omega_D.$$
We denote this set of functions $ \mathcal{AM}(\Omega_D)$.
\end{definition}

\begin{definition}[axially polyanalytic function of order 2]
\label{axpol}
	Let $f: \Omega_D \subset \mathbb{R}^4 \to \mathbb{H}$ be of axial type and of class $ \mathcal{C}^3(\Omega_D)$. If we apply the conjugate Fueter operator $ \mathcal{\overline{D}}$ to \eqref{slice} we get
	$$\breve{f}^0(q)= \mathcal{\overline{D}}f(x)\qquad \hbox{on} \quad \Omega_D,$$ 
	which is an axially polyanalytic function of order two, by the Fueter mapping theorem, i.e.,
	$$ \mathcal{D}^2 \breve{f}^0(x)=0	\qquad \hbox{on} \quad \Omega_D.$$
	We denote this set of functions $\mathcal{AP}_2(\Omega_D)$.
\end{definition}

\begin{definition}[axially harmonic functions]
If we apply the Fueter operator to an axial function $f$ of class $ \mathcal{C}^3(\Omega_D)$ we get
$$ \tilde{f}(q):= \mathcal{D} f(q)\qquad \hbox{on} \quad \Omega_D,$$
which is an axially harmonic function by the Fueter mapping theorem, i.e.,
$$ \Delta  \tilde{f}(q)=0 \qquad \hbox{on} \quad \Omega_D.$$
We denote this set of functions $ \mathcal{AH}(\Omega_D)$.
\end{definition}

Now we introduce the slice hyperholomorphic Cauchy kernels.
\begin{definition}
	Let $s,\, q\in \mathbb H$ with $q\notin [s]$ then we define 
	$$ \PRes_{s}(q)^{-1}:=(q^2-2\Re(s)q+|s|^2)^{-1} ,\ \ \ \
	\PRes_{c,s}(q)^{-1}:=(s^2-2\Re(q)s+|q|^2)^{-1},
	$$
	that are called pseudo Cauchy kernel and commutative pseudo Cauchy kernel, respectively.
\end{definition}
\begin{definition}\label{d1}
	Let $s,\, q\in \mathbb H$ with $q\notin [s]$ then
	\begin{itemize}
		\item We say that the left slice hyperholomorphic Cauchy kernel  $S^{-1}_L(s,q)$ is written in form I if
		$$S^{-1}_L(s,q):=\PRes_{s}(q)^{-1}(\overline s-q).$$
		\item We say that the right slice hyperholomorphic Cauchy kernel  $S^{-1}_R(s,q)$ is written in form I if
		$$S^{-1}_R(s,q):= (\overline s-q)\PRes_{s}(q)^{-1}.$$
		\item We say that the left slice hyperholomorphic Cauchy kernel  $S^{-1}_L(s,q)$ is written in form II if
		$$S^{-1}_L(s,q):=(s-\overline q) \PRes_{c,s}(q)^{-1}.$$
		\item We say that the right slice hyperholomorphic Cauchy kernel  $S^{-1}_R(s,q)$ is written in form II if
		$$S^{-1}_R(s,q):= \PRes_{c,s}(q)^{-1}(s- \overline q).$$
	\end{itemize}
\end{definition}
It is possible to prove that the left (resp. the right) slice Cauchy kernel is left (resp. right) slice hyperholomorphic in $q$ and right (resp. left) slice hyperholomoprhic in $s$ (see \cite[Lemma 2.1.27]{CGKBOOK} ). In this article, unless otherwise specified, we refer to $S^{-1}_L(s,q)$ and $S^{-1}_R(s,q)$  written in form II. 

We can state the Cauchy formulas for the slice hyperholomorphic functions.
\begin{theorem}[The Cauchy formulas for slice hyperholomorphic functions]
	\label{Cauchy}
	Let $U\subset\mathbb{H}$ be a bounded slice Cauchy domain, let $J\in\mathbb{S}$ and set  $ds_J=ds (-J)$.
	If $f$ is a left slice hyperholomorphic function on a set that contains $\overline{U}$ then
	\begin{equation}
		f(q)=\frac{1}{2 \pi}\int_{\partial (U\cap \mathbb{C}_J)} S_L^{-1}(s,q)\, ds_J\,  f(s),\qquad\text{for any }\ \  q\in U.
	\end{equation}
	If $f$ is a right slice hyperholomorphic function on a set that contains $\overline{U}$,
	then
	\begin{equation}\label{Cauchyright}
		f(q)=\frac{1}{2 \pi}\int_{\partial (U\cap \mathbb{C}_J)}  f(s)\, ds_J\, S_R^{-1}(s,q),\qquad\text{for any }\ \  q\in U.
	\end{equation}
	These integrals  depend neither on $U$ nor on the imaginary unit $J\in\mathbb{S}$.
\end{theorem}
Moreover, for slice hyperholomoprhic functions hold a version of the Cauchy integral theorem (see \cite[Theorem 2.1.21]{CGKBOOK}).
\begin{theorem}[Cauchy integral Theorem]\label{CIT}
	Let $U\subset \hh$ be open, let $J\in\mathbb S$, and let $f\in SH_L(U)$ and $g\in SH_R(U)$. Moreover, let $D_J\subset U\cap\cc_J$ be an open and bounded subset of the complex plane $\cc_J$ with $\overline D_J\subset U\cap\cc_J$ such that $\partial D_J$ is a finite union of piecewise continuously differentiable Jordan curves. Then
	$$\int_{\partial D_J} g(s)ds_J f(s)=0,$$
	where $ds_J=ds(-J)$.
\end{theorem}
By applying the Fueter map, namely the Laplace operator $\Delta$, to the second form of the slice Cauchy kernel we get the so called $F$-kernels (\cite{CSS}).
\begin{definition}
	\label{fres}
	Let $q$, $s \in \mathbb{H}$. We define for $ s \notin [q]$, the left $F$-kernel as
	\begin{equation}\label{lfk}
		F_{L}(s,q):= \Delta S^{-1}_L(s,q)=-4(s-\overline q)\PRes_{c,s}(q)^{-2},
	\end{equation}
	and the right $F$-kernel as
	\begin{equation}\label{rfk}
		F_R(s,q):=\Delta S^{-1}_R(s,q)=-4\PRes_{c,s}(q)^{-2} (s-\overline q).
	\end{equation}
\end{definition}
It is possible to prove that $F_L(s,q)$ (resp. $F_R(s,q)$) is left (resp. right) Fueter regular in $q$ and right (resp. left) slice hyperholomorphic in $s$ (see \cite{CGKBOOK}).
\begin{theorem}[The Fueter mapping theorem in integral form]
	\label{Fueter}
	Let $U\subset\mathbb{H}$ be a slice Cauchy domain, let $J\in\mathbb{S}$ and set  $ds_J=ds (-J)$.
	\begin{itemize}
		\item
		If $f$ is a left slice hyperholomorphic function on a set $W$, such that $\overline{U} \subset W$, then
		the left Fueter regular function  $\breve{f}(q)=\Delta f(q)$
		admits the integral representation
		\begin{equation}\label{FuetLSEC}
			\breve{f}(q)=\frac{1}{2 \pi}\int_{\partial (U\cap \mathbb{C}_J)} F_L(s,q)ds_J f(s).
		\end{equation}
		\item
		If $f$ is a right slice hyperholomorphic function on a set $W$, such that $\overline{U} \subset W$, then
		the right Fueter regular function $\breve{f}(q)=\Delta f(q)$
		admits the integral representation
		\begin{equation}\label{FuetRSCE}
			\breve{f}(q)=\frac{1}{2 \pi}\int_{\partial (U\cap \mathbb{C}_J)} f(s)ds_J F_R(s,q).
		\end{equation}
	\end{itemize}
	The integrals  depend neither on $U$ and nor on the imaginary unit $J\in\mathbb{S}$.
\end{theorem}
The main advantage to have the Fueter mapping theorem in integral form is that it is possible to obtain a monogenic function by computing the integral of a  suitable slice hyperholomorphic function. By applying only the Fueter operator $\mathcal D$ to the slice Cauchy kernel you get the pseudo Cauchy kernel (see \cite{CDPS}). 
\begin{theorem}
	\label{res2bis}
	Let $s$, $q \in \mathbb{H}$, be such that $s \notin [q]$ then
	$$\mathcal{D} S^{-1}_L(s,q)=-2\mathcal{Q}_{c,s}(q)^{-1}$$
	and
	$$S^{-1}_R(s,q) \mathcal{D}=-2\mathcal{Q}_{c,s}(q)^{-1} .$$
\end{theorem}
We note that the pseudo Cauchy kernels are harmonic functions in $q$ and, respectively, right and left slice hyperholomoprhic in $s$.

By applying the Fueter operator $\mathcal D$ to the Cauchy formulas in Theorem \ref{Cauchy} we the following result(see \cite{CDPS}).

\begin{theorem}[Integral representation of axially harmonic functions]
	\label{qthe}
	Let $W \subset \mathbb{H}$ be an open set. Let $U$ be a slice Cauchy domain such that $\overline{U} \subset W$. Then for $J \in \mathbb{S}$ and $ds_J=ds(-J)$ we have:
	\begin{itemize}
		\item[1)]If $f \in SH_L(W)$, then the function $ \tilde{f}(q)=\mathcal{D} f(q)$ is harmonic and it admits the following integral representation
		\begin{equation}
			\label{qform}
			\tilde{f}(q)=- \frac{1}{\pi} \int_{\partial(U \cap \mathbb{C}_J)} \mathcal{Q}_{c,s}(q)^{-1}ds_J f(s),\ \ \ \ q\in  U.
		\end{equation}
		\item[2)] If $f \in SH_R(W)$, then the function $ \tilde{f}(q)= f(q)\mathcal{D}$ is harmonic and it admits the following integral representation
		\begin{equation}
			\tilde{f}(q)=- \frac{1}{\pi} \int_{\partial(U \cap \mathbb{C}_J)} f(s) ds_J \mathcal{Q}_{c,s}(q)^{-1},\ \ \ \ q\in  U.
		\end{equation}
	\end{itemize}
	The integrals depend neither on $U$ nor on the imaginary unit $J  \in \mathbb{S}$.
\end{theorem}

\subsection{The $S$-functional calculus, the $F$-functional calculus and the $Q$-functional calculus}
Let $X$ be a two sided quaternionic Banach module of the form $X= X_{\mathbb{R}} \otimes \mathbb{H}$, where $X_{\mathbb{R}}$ is a real Banach space. In this paper we consider $\mathcal{B}(X)$ the Banach space of all bounded right linear operators acting on $X$.
\\In the sequel we will consider bounded operators of the form $T=T_0+T_1e_1+T_2e_2+T_3e_3$, with commuting components $T_{i}$ acting on a real vector space $X_{\mathbb{R}}$, i.e., $T_i \in \mathcal{B}(X_{\mathbb{R}})$ for $i=0,1,2,3$. The subset of $ \mathcal{B}(X)$ given by the operators $T$ with commuting components $T_i$ will be denoted by $ \mathcal{BC}(X)$.

Now let $T:X \to X$ be a right (or left) linear operator. We give the following.
\begin{definition}
	Let $T\in \mathcal B(X)$. For $s\in\mathbb H$ we set
	$$ \mathcal Q_s(T):=T^2-2\Re(s)T+|s|^2\mathcal I. $$
	We define the $S$-resolvent set $\rho_S(T)$ of $T$ as
	$$\rho_S(T):=\{s\in\mathbb H: \, \mathcal Q_s(T)^{-1} \in \mathcal B(X)\},$$
	and we define the $S$-spectrum $\sigma_S(T)$ of $T$ as
	$$ \sigma_S(T):=\mathbb H\setminus \rho_S(T). $$
\end{definition}
For $s\in\rho_S(T)$, the operator $\mathcal Q_s(T)^{-1}$ is called the pseudo $S$-resolvent operator of $T$ at $s$.

\begin{theorem}
	Let $T\in\mathcal B(X)$ and $s\in\mathbb H$ with $\|T\|< |s|$. Then we have
	$$ \sum_{n=0}^{\infty}T^n s^{-n-1}=-\mathcal{Q}_s(T)^{-1}(T-\overline s\mathcal I),$$
	and
	$$ \sum_{n=0}^{\infty}s^{-n-1}T^n=-(T-\overline s\mathcal I)\mathcal{Q}_s(T)^{-1}. $$
\end{theorem}
According to the left or right slice hyperholomorphicity, there exist two different resolvent operators.
\begin{definition}[$S$-resolvent operators]

	Let $T \in \mathcal{B}(X)$ and $s\in\mathbb \rho_{S}(T)$. Then the left $S$-resolvent operator is defined as
	$$ S^{-1}_L(s,T):=-\mathcal{Q}_s(T)^{-1}(T-\overline s\mathcal I), $$
	and the right $S$-resolvent operator is defined as
	$$ S^{-1}_R(s,T):=-(T-\overline s\mathcal I)\mathcal{Q}_s(T)^{-1}. $$
\end{definition}

In order to give the definition of the $S$-functional calculus we need to introduce some notations. Let $T \in \mathcal{B}(X)$. We denote by $SH_L(\sigma_S(T))$, $SH_R(\sigma_S(T))$ and $N(\sigma_S(T))$ the sets of all left, right and intrinsic slice hyperholomorphic functions, respectively, with $ \sigma_S(T) \subset dom(f)$.

\begin{definition}[$S$-functional calculus]
	\label{Sfun}
	Let $T \in \mathcal{B}(X)$. Let $U$ be a slice Cauchy domain that contains $\sigma_S(T)$  and $\overline{U}$ is contained in the domain of $f$.  Set $ds_J=-dsJ$ for $J\in \mathbb{S}$ so we define
	\begin{equation}
		\label{Scalleft}
		f(T):={{1}\over{2\pi }} \int_{\partial (U\cap \mathbb{C}_J)} S_L^{-1} (s,T)\  ds_J \ f(s), \ \ {\rm for\ every}\ \ f\in SH_L(\sigma_S(T))
	\end{equation}
	and
	\begin{equation}
		\label{Scalright}
		f(T):={{1}\over{2\pi }} \int_{\partial (U\cap \mathbb{C}_J)} \  f(s)\ ds_J
		\ S_R^{-1} (s,T),\ \  {\rm for\ every}\ \ f\in SH_R(\sigma_S(T)).
	\end{equation}
\end{definition}
The definition of $S$-functional calculus is well posed since the integrals in (\ref{Scalleft}) and (\ref{Scalright}) depend neither on $U$ and nor on the imaginary unit $J\in\mathbb{S}$, see \cite{CGKBOOK, JFACSS}. In Section 4 it will be useful the following resolvent equation for the $S$-functional calculus (see \cite{ACGS15}).
\begin{theorem}
	\label{Sres}
	Let $T \in \mathcal{BC}(X)$ such that it commutes with $T$, then we have
	\begin{equation}\label{sresc}
		\begin{split}
			S^{-1}_R(s,T)S^{-1}_L(p,T)=& [\left(S^{-1}_R(s,T)-S^{-1}_L(p,T)\right)p+\\
			&- \bar{s}\left(S^{-1}_R(s,T)-S^{-1}_L(p,T)\right)] \mathcal{Q}_s(p)^{-1},
		\end{split}
	\end{equation}
	where $ \mathcal{Q}_s(p):= p^2-2s_0 p+|s|^2$.
\end{theorem}

Now we want to introduce the $F$-functional calculus. Let us consider $T=T_{0}+ T_1e_1+T_2e_2+T_3e_3$ such that $T \in \mathcal{BC}(X)$.
\begin{definition}\label{QCS}
	Let $T \in \mathcal{BC}(X)$. For $s \in \mathbb{H}$ we set
	$$\mathcal{Q}_{c,s}(T)=s^2-s(T+\overline{T})+T \overline{T},$$
	where $ \overline {T}=T_{0}- T_1e_1-T_2e_2-T_3e_3$. We define the $ F$-resolvent set as
	$$ \rho_F(T)=\{s\in\mathbb H:\, \mathcal Q_{c,s}(T)^{-1}\in\mathcal B(X)\}. $$
	Moreover, we define the $F$-spectrum of $T$ as
	$$ \sigma_{F}(T)= \mathbb{H}\setminus \rho_F(T).$$
\end{definition}
By \cite[Prop. 4.14]{CSS} it turns out that the $F$-spectrum is the commutative version of the $S$-spectrum, i.e., we have
$$ \sigma_{F}(T)=\sigma_{S}(T), \qquad T \in \mathcal{BC}(X),$$
and consequently $\rho_F(T)=\rho_S(T)$.

For $s \in \rho_{S}(T)$ the operator $ \mathcal{Q}_{c,s}(T)^{-1}$ is called the commutative pseudo $S$-resolvent operator of $T$. It is possible to define a commutative version of the $S$-functional calculus (see \cite{CG}).

\begin{theorem}
	Let $T \in \mathcal{BC}(X)$ and $s \in \mathbb{H}$ be such that $\| T\| < s$. Then
	$$ \sum_{m=0}^\infty T^m s^{-1-m}=(s \mathcal{I}-\overline T)\mathcal Q_{c,s}(T)^{-1},$$
	and
	$$ \sum_{m=0}^\infty s^{-1-m}T^m =\mathcal Q_{c,s}(T)^{-1}(s\mathcal{I}-\overline T).$$
\end{theorem}
\begin{definition}
	Let $T \in \mathcal{BC}(X)$ and $s \in \rho_{S}(T)$. We define the left commutative $S$-resolvent operator as
	$$ S^{-1}_L(s,T)=(s \mathcal{I}-\overline T)\mathcal Q_{c,s}(T)^{-1}, $$
	and the right commutative $S$-resolvent operator as
	$$
	S^{-1}_R(s,T)=\mathcal Q_{c,s}(T)^{-1}(s\mathcal{I}-\overline T).
	$$
\end{definition}

For the sake of simplicity we denote the commutative version of the $S$-resolvent operators with the same symbols used for the noncommutative ones. It is possible to define the $S$-functional calculus as done in Definition \ref{Sfun}.
In the sequel, when we deal with the $S$-resolvent operators, we intend their commutative version. Now we give the definition of the $F$-functional calculus.	
\begin{definition}[$F$-resolvent operators] Let $T \in \mathcal{BC}(X)$. We define the left $F$-resolvent operator as
	$$ F_{L}(s,T)=-4(s\mathcal{I}- \overline{T}) \mathcal{Q}_{c,s}(T)^{-2}, \qquad s \in \rho_{S}(T),$$
	and the right $F$-resolvent operator as
	$$ F_{R}(s,T)=-4\mathcal{Q}_{c,s}(T)^{-2}(s\mathcal{I}- \overline{T}) , \qquad s \in \rho_{S}(T).$$
\end{definition}
With the above definitions and Theorem \ref{Fueter} at hand, we can recall the $F$-functional calculus. It was first introduced in \cite{CSS} and then investigated in \cite{CDS, CG,CS}.

\begin{definition}[The $F$-functional calculus for bounded operators]
	Let $U$ be a slice Cauchy domain that contains $\sigma_S(T)$  and $\overline{U}$ is contained in the domain of $f$.
	Let $T= T_1e_1+T_2e_2+T_3e_3 \in\mathcal{BC}(X)$, assume that the operators $T_{\ell}$, $\ell=1,2,3$ have real spectrum and set $ds_J=ds(-J)$, where $J\in \mathbb{S}$.
	For any function $f\in SH_L(\sigma_S(T))$, we define
	\begin{equation}\label{DefFCLUb}
		\breve{f}(T):=\frac{1}{2\pi}\int_{\pp(U\cap \mathbb{C}_J)} F_L(s,T) \, ds_J\, f(s).
	\end{equation}
	For any $f\in SH_R(\sigma_S(T))$, we define
	\begin{equation}\label{SCalcMON}
		\breve{f}(T):=\frac{1}{2\pi}\int_{\pp(U\cap \mathbb{C}_J)} f(s) \, ds_J\, F_R(s,T).
	\end{equation}
	
\end{definition}
The definition of the $F$-functional calculus is well posed since
the integrals in (\ref{DefFCLUb}) and (\ref{SCalcMON}) depend neither on $U$ and nor on the imaginary unit $J\in\mathbb{S}$.

Another important tool is to write the commutative pseudo $S$-resolvent operator in terms of the $F$-resolvents, see \cite[Theorem 7.3.1]{CGKBOOK}
\begin{theorem}
	\label{qf}
	Let $T \in \mathcal{BC}(X)$ and let $s \in \rho_S(T)$. The $ F$ resolvent operators satisfy the equations
	$$  F_L(s,T)s-T F_L(s,T)=-4 \mathcal{Q}_{c,s}(T)^{-1},$$
	and
	$$  s F_R(s,T)- F_{R}(s,T)T=-4 \mathcal{Q}_{c,s}(T)^{-1}.$$
\end{theorem}

We conclude this section with the definition of the $Q$-functional calculus (harmonic functional calculus). This is crucial to get a product rule for the $F$-functional calculus (see \cite{CDPS}). 

\begin{definition}[$Q$-functional calculus on the $S$-spectrum]
	\label{qfun1}
	Let $T \in \mathcal{BC}(X)$ and set $ds_J=ds(-J)$ for $J \in \mathbb{S}$. For every function $\tilde{f}=\mathcal{D} f$ with $f \in SH_{L}(\sigma_S(T))$, we set
	\begin{equation}
		\label{inte1b}
		\tilde{f}(T):= - \frac{1}{\pi} \int_{\partial(U \cap \mathbb{C}_J)} \mathcal{Q}_{c,s}(T)^{-1} ds_J f(s),
	\end{equation}
	where $U$ is an arbitrary bounded slice Cauchy domain with $\sigma_{S}(T) \subset U$ and $ \overline{U} \subset dom(f)$ and $J \in \mathbb{S}$ is an arbitrary imaginary unit.
	\\ For every function $\tilde{f}= f\mathcal{D}$ with $f \in SH_{R}(\sigma_S(T))$, we set
	\begin{equation}
		\label{inte2b}
		\tilde{f}(T):=- \frac{1}{\pi} \int_{\partial (U \cap \mathbb{C}_J)} f(s) ds_J \mathcal{Q}_{c,s}(T)^{-1},
	\end{equation}
	where $U$ and $J$ are as above.
\end{definition}
In the previous definitions of the functional calculi we have always a right version and a left version. However, if we consider an intrinsic function the two versions coincide.
\begin{theorem}
	\label{intri}
	Let $T\in\mathcal B(X)$. If $f\in N(\sigma_S(T))$, then we have
	$$ \frac 1{2\pi}\int_{\partial(U\cap\cc_J)} S^{-1}_L(s,T)\, ds_J\, f(s)=\frac 1{2\pi}\int_{\partial(U\cap \cc_J)} f(s)\, ds_J\, S^{-1}_R(s,T),$$

	$$ -\frac 1{\pi}\int_{\partial(U\cap\cc_J)} \mathcal Q_{c,s}(s,T)\, ds_J\, f(s)=-\frac 1{\pi}\int_{\partial(U\cap \cc_J)} f(s)\, ds_J\, \mathcal Q_{c,s} (s,T),$$

	$$ \frac 1{2\pi}\int_{\partial(U\cap\cc_j)} F_L(s,T)\, ds_J\, f(s)=\frac 1{2\pi}\int_{\partial(U\cap \cc_J)} f(s)\, ds_J\, F_R(s,T).$$
\end{theorem}
\begin{theorem}[The product rule for the $F$-functional calculus]
	
	Let $T \in \mathcal{BC}(X)$ and assume $f \in N(\sigma_S(T))$ and $g \in SH_L(\sigma_S(T))$ then we have
	\begin{equation}
		\label{Pr0}
		\Delta (fg)(T)=(\Delta f)(T) g(T)+f(T) (\Delta g)(T)- (\mathcal{D}f)(T) (\mathcal{D} g)(T).
	\end{equation}
\end{theorem}

\section{The polyanalytic functional calculus of order $2$ on the $S$-spectrum}
In \cite{DP}, as far the authors know, a polyanalytic functional calculus of order $2$ is defined for the first time. In this section we recall the main results of \cite{DP} and we will prove further properties for this functional calculus. It is based on applying the conjugate of the Fueter operator to the slice hyperholomorphic Cauchy kernels, as illustrated in diagram \eqref{diagg}.

\begin{proposition}
	Let $q$, $s \in \mathbb{H}$ be such that $ x \notin [s]$. Let $S^{-1}_L(s,q)$ and $S^{-1}_R(s,q)$ be the slice hyperholomorphic Cauchy kernels written in form II. Then
	\begin{itemize}
		\item The function $ \mathcal{\overline{D}}S^{-1}_L(s,q)$ is a left polyanalytic function of order 2 in the variable $q$ and right slice hyperholomorphic in $s$.
		\item The function $ S^{-1}_R(s,q)\mathcal{\overline{D}}$ is a right polyanalytic function of order 2 in the variable $q$ and left slice hyperholomorphic in $s$.
	\end{itemize}
\end{proposition}
In \cite{DP} we have provided direct computations for $ \mathcal{\overline{D}}S^{-1}_L(s,q)$ and $ S^{-1}_R(s,q)\mathcal{\overline{D}}$. We define the $ \mathcal{P}^L_2(s,q)$ and $ \mathcal{P}^R_2(s,q)$ kernels as
\begin{theorem}
	\label{res2}
	Let $q$, $s \in \mathbb{H}$. For $ x \notin [s]$ we define the $ \mathcal{P}^L(s,q)$ kernel as
	$$ \mathcal{P}^L_2(s,q):= \mathcal{\overline{D}}S^{-1}_L(s,q)=- F_L(s,q)s+q_0 F_L(s,q)$$
	and the $ \mathcal{P}^R(s,T)$ as 
	$$ \mathcal{P}^R_2(s,q):=S^{-1}_R(s,q) \mathcal{\overline{D}}=- sF_L(s,q)+q_0 F_L(s,q)$$
	where $F_L(s,q)$ and $F_R(s,q)$ are defined in \eqref{lfk} and \eqref{rfk}.
\end{theorem}
\begin{theorem}[Integral representation of axially polyanalytic functions of order $2$]
	\label{inrap}
	Let $W\subset \hh$ be an open set. Let $U$ be a slice Cauchy domain such that $\overline U\subset W$. Then for $J\in\mathbb S$ and $ds_J=ds(-J)$ we have 
	\begin{enumerate}
		\item if $f\in \mathcal{SH}_L(W)$, then the function $\breve f^0(q)=\overline\bigD f(q)$ is polyanalytic of order $2$ and it admits the following integral representation 
		\begin{equation}
			\label{star}
			\breve f^0(q)=-\frac 1{2\pi}\sum_{k=0}^1(-q_0)^k\int_{\partial(U\cap\cc_J)} F_L(s,q) s^{1-k} \, ds_J\, f(s)\quad \forall q\in U;
		\end{equation}
		\item if $f\in \mathcal{SH}_R(W)$, then the function $\breve f^0(q)= f(q) \overline\bigD$ is polyanalytic of order $2$ and it admits the following integral representation  
		\begin{equation}
			\label{start2}
			\breve f^0(q)=-\frac 1{2\pi}\sum_{k=0}^1(-q_0)^k\int_{\partial(U\cap\cc_J)} f(s) \, ds_J\,  s^{1-k} F_R(s,q)\quad \forall q\in U.
		\end{equation}
	\end{enumerate}
	The integrals depend neither on $U$ nor on the imaginary unit $J\in \mathbb S$.
\end{theorem}
In order to get an expansion in series of the $\mathcal{P}^L_2(s,q)$ and $ \mathcal{P}^R_2(s,q)$, in \cite{DP} we use the following result.
\begin{lemma}
	\label{res3b}
	For $n \geq 1$ we have
	\begin{equation}
		\label{F1}
		\mathcal{\overline{D}} q^n= 2 \left( n q^{n-1}+ \sum_{k=1}^n q^{n-k} \bar{q}^{k-1} \right).
	\end{equation}
	Moreover,
	\begin{equation}
		\label{comm}
		q^n\mathcal{\overline{D}}=\mathcal{\overline{D}} q^n.
	\end{equation}
\end{lemma}
In this paper we will show that we can prove Lemma \ref{res3b} by using the following result,  see \cite[Lemma 1]{B}.
\begin{lemma}
	\label{res4}
	Let $n \geq 2$ then we have
	\begin{equation}
		\mathcal{D}\overline{q}^n=(2n+2)\bar{q}^{n-1}+2q \sum_{k=0}^{n-2} \bar{q}^{n-k-2} q^k.
	\end{equation}
\end{lemma}
Before to show Lemma \ref{res3b} we need an auxiliary result.
\begin{lemma}
	\label{aux1}
	For $n \geq 1$ and $q \in \mathbb{H}$ we have
	$$\mathcal{D} q^n= \overline{\mathcal{D}} \overline{q}^n. $$
	\begin{proof}
		We show the formula by induction on $n$. For $n=1$ it is trivial. Now, we suppose that it holds for $n$ and we prove it for $n+1$. Since 
		$$ \bar{q}^{n+1}= \bar{q}^n( \bar{q}+q)- \bar{q}^{n-1}|q|^2,$$
		by the inductive hypothesis we have
		\begin{eqnarray*}
			\overline{\mathcal{D}}\bar{q}^{n+1}&=&\mathcal{\overline{D}}(\bar{q}^n)(\overline q+q)+2 \bar{q}^n - \overline{\mathcal{D}}(\bar{q}^{n-1}) |q|^2-2 \bar{q}^{n}\\
			&=& \mathcal{\overline{D}}(\bar{q}^n)(\overline q+q)- \overline{\mathcal{D}}(\bar{q}^{n-1}) |q|^2\\
			&=&  \mathcal{D}(q^n)(q+ \bar{q})- \mathcal{D}(q^{n-1})|q|^2\\
			&=& \mathcal{D}q^{n+1}.
		\end{eqnarray*}
	\end{proof}
\end{lemma}
\begin{proof}[Proof of Lemma \ref{res3b}]
	By Lemma \ref{aux1} we can write
	$$ \overline{\mathcal{D} \bar{q}^n}=\mathcal{\overline{D}} q^n.$$
	Finally, by Lemma \ref{res4}, we get
	\begin{eqnarray*}
		\mathcal{\overline{D}} q^n=\overline{\mathcal{D} \bar{q}^n}&=& \overline{(2n+2)\bar{q}^{n-1}+2q \sum_{k=0}^{n-2} \bar{q}^{n-2-k} q^k}\\
		&=& 2 \left((n+1)q^{n-1}+\sum_{k=0}^{n-2} q^{n-2-k} \bar{q}^{k+1}\right)\\
		&=& 2 \left((n+1)q^{n-1}+\sum_{k=2}^{n} q^{n-k} \bar{q}^{k-1}\right)\\
		&=& 2 \left( n q^{n-1}+ \sum_{k=1}^n q^{n-k} \bar{q}^{k-1} \right).
	\end{eqnarray*}
\end{proof}
In \cite{DP}, we showed that Lemma \ref{res3b} was crucial to get an expansion in series of $ \mathcal{P}_L(s,q)$ and $ \mathcal{P}_R(s,q)$.
\begin{lemma}
	For $q$, $s \in \mathbb{H}$ such that $|q| < |s|$, we have
	$$ \mathcal{P}^L_2(s,q)= 2 \sum_{n=1}^\infty  \left(n q^{n-1} + \sum_{j=1}^{n}  q^{n-j} \bar{q}^{j-1}\right)s^{-1-n},$$
	and
	$$ \mathcal{P}^R_2(s,q)= 2 \sum_{n=1}^\infty s^{-1-n} \left(n q^{n-1} + \sum_{j=1}^{n}  q^{n-j} \bar{q}^{j-1}\right).$$
\end{lemma}
Now, we have all the tools to introduce the $P_2$-functional calculus.
\begin{definition}\label{dbars}
	Let $T=T_0+\sum_{i=1}^3e_i T_i\in\mathcal B\mathcal C(X)$, $s\in\mathbb H$, we define the left $\overline{ \mathcal D}$-kernel operator as
	$$\mathcal{P}^L_2(s,T)=-F_L(s,T)s+T_0F_L(s,T)$$
	and the right  $\overline{ \mathcal D}$-kernel operator as
	$$\mathcal{P}^R_2(s,T)= -sF_R(s,T)+T_0F_R(s,T).$$
\end{definition}
\begin{definition}[Polyanalytic functional calculus of order $2$ on the $S$-spectrum]
	\label{qfun}
	Let $T \in \mathcal{BC}(X)$ and set $ds_J=ds(-J)$ for $J \in \mathbb{S}$. For every function $\breve{f}^\circ=\overline{\mathcal{D}} f$ with $f \in SH_{L}(\sigma_S(T))$, we set
	\begin{equation}
		\label{inte1}
		\breve{f}^\circ(T):=  \frac{1}{2\pi} \int_{\partial(U \cap \mathbb{C}_J)} \mathcal P^L_2(s,T) ds_J f(s),
	\end{equation}
	where $U$ is an arbitrary bounded slice Cauchy domain with $\sigma_{S}(T) \subset U$ and $ \overline{U} \subset dom(f)$. Moreover, $J \in \mathbb{S}$ is an arbitrary imaginary unit.
	\\ For every function $\breve{f}^\circ= f\overline{\mathcal{D}}$ with $f \in SH_{R}(\sigma_S(T))$, we set
	\begin{equation}
		\label{inte2}
		\breve{f}^\circ(T):= \frac{1}{2\pi} \int_{\partial (U \cap \mathbb{C}_J)} f(s) ds_J \mathcal P^R_2(s,T),
	\end{equation}
	where $U$ and $J$ are as above.
\end{definition} 
In \cite{DP} the authors showed that it is possible to write the right and the left $\overline{\mathcal D}$-kernels operators in terms of $T$ and $\bar T$.

\begin{proposition}\label{p2}
	Let $T=T_0+\sum_{i=1}^3 e_iT_i\in \mathcal{BC}(X)$, $s\in\hh$ and $\|T\|<|s|$, we can write the left $\overline{ \mathcal D}$-kernel operator as
	\begin{equation}\label{sl2}
		\mathcal P^L_2(s,T)=2\sum_{n=1}^\infty\left(nT^{n-1}+\sum_{k=1}^n T^{n-k}\bar T^{k-1}\right)s^{-1-n},
	\end{equation}
	and the right $\overline{ \mathcal D}$-kernel operator as
	\begin{equation}\label{sr2}
		\mathcal P^R_2(s,T)=2\sum_{n=1}^\infty s^{-1-n} \left(nT^{n-1}+\sum_{k=1}^nT^{n-k}\bar T^{k-1}\right).
	\end{equation} 
\end{proposition}
Finally, we study the regularity of the $ \mathcal{\overline{D}}$-kernel operators.
\begin{lemma}
	Let $T\in\mathcal{BC}(X)$. The left (resp. right) $\overline {\mathcal D}$-resolvent operator $\mathcal P^L_2(s,T)$ (resp. $\mathcal P^R_2(s,T)$), is a $\mathcal B(X)$-valued right (resp. left) slice hyperholomorphic function of the variable $s$ in $\rho_S(T)$. 
\end{lemma}
We recall and prove the following result for the sake of completeness.

\begin{theorem}
The $P_2$-functional calculus on the $S$-spectrum is well-defined, i.e., the integrals in \eqref{inte1} and \eqref{inte2} depend neither on the imaginary unit $J \in \mathbb{S}$ nor on the slice Cauchy domain $U$.
\end{theorem}
\begin{proof}
	Here we show only the case $ \breve{f}^\circ=\overline{\mathcal{D}}f$ with $f \in SH_L(\sigma_S(T))$, since the other one follows by analogous arguments.
	\\ Since $ \mathcal P^L_2(s,T) $ is a $B(X)$-valued right slice hyperholomorphic function in $s$ and $f$ is left slice hyperholomorphic, the independence from the set $U$ follows by the Cauchy integral formula, see Theorem \ref{Cauchy} and Theorem \ref{CIT}.
	\\ Now, we want to show the independence from the imaginary unit.	
	Let us consider two imaginary units $J$, $I \in \mathbb{S}$ with $J \neq I$ and two bounded slice Cauchy domains $U_q$, $U_s$ with $ \sigma_{S}(T) \subset U_q$, $\overline{U}_q \subset U_s$ and $\overline{U}_s \subset dom(f)$.
	Then every $s \in \partial (U_s \cap \mathbb{C}_J)$ belongs to the unbounded slice Cauchy domain $\mathbb{H}\setminus  U_q $.
	Recall that $\mathcal P^L_2(q,T) $ is right slice hyperholomorphic on $\rho_S(T)$, also at infinity, since $ \lim_{q \to + \infty} \mathcal P^L_2(q,T)=0$. Thus the Cauchy formula implies
	\begin{eqnarray}
		\nonumber
		\mathcal P^L_2(s,T) &=& \frac{1}{2 \pi} \int_{\partial \left((\mathbb{H} \setminus U_q) \cap \mathbb{C}_I \right)} \mathcal P^L_2(q,T)  dq_I S^{-1}_R(q,s)\\
		\label{inte3}
		&=& \frac{1}{2 \pi} \int_{\partial (U_q \cap \mathbb{C}_I)} \mathcal P^L_2(q,T)  dq_I S^{-1}_L(s,q).
	\end{eqnarray}
	The last equality is due to the fact that $ \partial \left((\mathbb{H} \setminus U_q) \cap \mathbb{C}_I \right)=-\partial (U_q \cap \mathbb{C}_I)$ and $S^{-1}_R(q,s)=-S^{-1}_L(s,q).$ Combining \eqref{inte1} and \eqref{inte3} we get
	\begin{eqnarray*}
		\breve{f}^\circ(T)&=& \frac{1}{ 2\pi} \int_{\partial(U_s \cap \mathbb{C}_J)} \mathcal P^L_2(s,T)  ds_J f(s)\\
		&=& \frac{1}{ 2\pi} \int_{\partial(U_s \cap \mathbb{C}_J)} \left( \frac{1}{2 \pi} \int_{\partial(U_q \cap \mathbb{C}_I)} \mathcal P^L_2(q,T)  dq_I S_{L}^{-1}(s,q)\right) ds_J f(s).
	\end{eqnarray*}
	Due to Fubini's theorem we can exchange the order of integration and by the Cauchy formula we obtain
	\begin{eqnarray*}
		\breve{f}^\circ(T)&=&  \frac{1}{ 2\pi} \int_{\partial (U_q \cap \mathbb{C}_I)} \mathcal P^L_2(q,T) dq_I  \left( \frac{1}{2 \pi} \int_{\partial (U_s \cap \mathbb{C}_J)}   S_{L}^{-1}(s,q) ds_J f(s) \right)\\
		&=&  \frac{1}{2\pi} \int_{\partial(U_q \cap \mathbb{C}_I)} \mathcal P^L_2(q,T)  dq_I f(q).
	\end{eqnarray*}
	This proves the statement.
\end{proof}
The following result is also important to have a well posed functional calculus. 
\begin{theorem}
	\label{ps2}
	Let $ U$ be a slice Cauchy domain. If $f,g\in SH_L(U)$ (resp. $f,g\in SH_R(U)$) and $\overline{\mathcal{D}}f=\overline{\mathcal{D}}g$  (resp. $f\overline{\mathcal{D}}=g\overline{\mathcal{D}}$) then for any $T\in\mathcal{BC}(X)$ such that $T= T_0e_0+T_1e_1+T_2 e_2$, and assuming that the operators $T_{\ell}$, $\ell=0,1,2$, have real spectrum, we have
	$$\breve f^0(T)=\breve g^0(T).$$ 
\end{theorem}
In order to prove the previous theorem we need some auxiliary results. First of all, we have to study the following sets
$$ (\ker{\overline{\mathcal{D}}})_{SH_L(\Omega)}:=\{f\in SH_L(\Omega): \overline{\mathcal{D}}(f)=0\}\quad\textrm{and}\quad (\ker{\overline{\mathcal{D}}})_{SH_R(\Omega)}:=\{f\in SH_R(\Omega): (f)\overline{\mathcal{D}}=0\} .$$ 
It is necessary to study these sets because in the hypothesis of Theorem \ref{ps2} we have $\overline{\mathcal{D}}(f-g)=0$ (resp. $(f-g)\overline{\mathcal{D}}=0$).

\begin{theorem}\label{Tcost} Let $\Omega$ be a connected slice Cauchy domain of $\mathbb H$, then
	\begin{eqnarray*}
		(\ker{\overline{\mathcal{D}}})_{SH_L(\Omega)}&=&\{f\in SH_L(\Omega): f\equiv \alpha\quad\textrm{for some $\alpha\in\mathbb H$}\}\\
		&=&\{f\in SH_R(\Omega): f\equiv \alpha\quad\textrm{for some $\alpha\in\mathbb H$}\}=(\ker{\overline{\mathcal{D}}})_{SH_R(\Omega)}.
	\end{eqnarray*}
	
\end{theorem}
\begin{proof}
	We prove the result in the case $f\in SH_L(\Omega)$ since the case $f\in SH_R(\Omega)$ follows by similar arguments. We proceed by double inclusion. The fact that
	$$
	(\ker{\overline{\mathcal{D}}})_{SH_L(\Omega)}\supseteq\{f\in SH_L(\Omega): f\equiv \alpha\quad\textrm{for some $\alpha\in\mathbb H$}\}
	$$ is obvious. The other inclusion can be proved observing that if $f\in (\ker{\overline{\mathcal{D}}})_{SH_L(\Omega)}$, after a change of variable if needed, there exists $r>0$ such that the function $f$ can be expanded in a convergent series at the origin
	$$f(q)=\sum_{k=0}^{\infty}q^k \alpha_k\quad\textrm{for $(\alpha_k)_{k\in\mathbb N_0}\subset\mathbb H$ and for any $q\in B_r(0)$}$$
	where $B_r(0)$ is the ball centered at $0$ and of radius $r$. By Lemma \ref{res4}, we have
	\begin{equation}\label{novaeq1}
		0=\overline{\mathcal{D}}f(q)\equiv\sum_{k=1}^{\infty} \overline{\mathcal{D}}(q^k)\alpha_k=2\sum_{k=1}^\infty\left(kq^{k-1}+\sum_{s=1}^k q^{k-s}\bar q^{s-1}\right) \alpha_k,\quad \forall q\in B_r(0).
	\end{equation}
	If we restrict the previous series in \eqref{novaeq1} in a neighborhood $U$ of $0$ of the real line we get
	$$0=\sum_{k=1}^{\infty}q_0^{k-1}\alpha_k\quad \forall\, q_0\in U$$
	and this implies 
	$$\alpha_k=0,\quad \forall k\geq 1.$$
	Thus $f(q)\equiv \alpha_0$ in $B_r(0)$ and since $\Omega$ is connected $f(q)\equiv\alpha_0$ for any $q\in\Omega$.
\end{proof}

To define a monogeinc functional McIntosh and collaborators, see \cite{J}, had as hypothesis that the component $T_0$ of the operator $T=T_0+T_1e_1+T_2e_2+T_3e_3$ is zero. However, it is possible to set zero a different component of the operator $T$. In a polyanalytic functional calculus is not convenient to have $T_0=0$, due to the left and right structure of the $ \mathcal{\overline{D}}$-kernel (see Definition \ref{dbars}). For this reason, in the present work we impose the last component of the operator $T$ to be zero, i.e., $T_3=0$.  

\begin{lemma}
	\label{mono}
	Let $T \in \mathcal{BC}(X)$ be such that $T= T_0e_0+ T_1e_1+ T_2 e_2$, and assume that the operators $T_{\ell}$, $\ell=0,1,2$, have real spectrum. Let $G$ be a bounded slice Cauchy domain such that $(\partial G) \cap \sigma_{S}(T)= \emptyset$. For every $J \in \mathbb{S}$ we have
	\begin{equation}\label{novaeq2}
		\int_{\partial{(G \cap \mathbb{C}_J)}} \mathcal P^L_2(s,T) ds_J=0\quad\textrm{and}\quad \int_{\partial{(G \cap \mathbb{C}_J)}} ds_J \mathcal P^R_2(s,T) =0.
	\end{equation}
\end{lemma}
\begin{proof}
We prove only the first equation of \eqref{novaeq2}, since the other one follows by similar computations. Since $\Delta(1)=0$ and $\Delta(q)=0$, by Theorem \ref{Fueter} we have
	\begin{equation}\label{nova}
		\int_{\partial{(G \cap \mathbb{C}_J)}} F_L(s,q) ds_J=\Delta(1)=0,
	\end{equation}
	and
	\begin{equation}\label{nova_bis}
		\int_{\partial{(G \cap \mathbb{C}_J)}} F_L(s,q) ds_Js =\Delta(q)=0,
	\end{equation}
	for all $q \notin \partial G$ and $J \in \mathbb{S}$. By the monogenic functional calculus of McIntosh and collaborators, see \cite{J}, we have
	$$ F_L(s,T)= \int_{\partial \Omega} G(\omega,T) \mathbf{D}\omega F_L(s,\omega),$$
	where $\mathbf{D}\omega$ is a suitable differential form, the open set $ \Omega$ contains the left spectrum of $T$ and $G(\omega,T)$ is the Fueter resolvent operator. By Definition \ref{dbars} we can write
	$$\mathcal P^L_2(s,T)=-F_L(s,T)s+q_0 F_L(s,T)),$$ 
	thus  we have
	\begin{eqnarray*}
		&& \int_{\partial{(G \cap \mathbb{C}_J)}} \mathcal P^L_2(s,T) ds_J=-\int_{\partial{(G \cap \mathbb{C}_J)}}  F_L(s,T)s-T_0F_L(s,T) ds_J\\
		&&=- \left( \int_{\partial{(G \cap \mathbb{C}_J)}}  \int_{\partial\Omega} G(\omega, T)\mathbf{D}\omega F_L(s,\omega)s\, ds_J - T_0\int_{\partial{(G \cap \mathbb{C}_J)}}\int_{\partial\Omega} G(\omega,T)\mathbf{D}\omega F_L(s,\omega) ds_J\right)\\
		&&= -\frac 14\left( \int_{\partial \Omega} G(\omega,T) \mathbf{D}\omega \left(\int_{\partial{(G \cap \mathbb{C}_J)}} F_L(s,\omega) ds_J s \right)-T_0\int_{\partial \Omega} G(\omega,T) \mathbf{D}\omega \left(\int_{\partial{(G \cap \mathbb{C}_J)}} F_L(s,\omega)  ds_J\right)\right)\\
		&&= 0
	\end{eqnarray*}
	where the second equality is a consequence of the Fubini's Theorem and the last equality is a consequence of formulas \eqref{nova} and \eqref{nova_bis}. 
\end{proof}
\begin{proof}[Proof of Theorem \ref{ps2}]
	We prove the theorem when $f,g\in SH_L(\Omega)$. The case of $f,g\in SH_R(\Omega)$ follows by similar arguments. We divide the proof in two cases.
	\newline
	\newline
	\fbox{$U$ is connected}
	\newline
	\newline
	By definition of the $P_2$-functional calculus on the $S$-spectrum, see Definition \ref{qfun}, we have
	$$
	\breve f^0(T)-\breve g^0(T)=\frac 1{2\pi}\int_{\partial (U\cap\cc_J)}\mathcal P^L_2(s,T)ds_J(f(s)-g(s)).
	$$
	Since $\mathcal P^L_2(s,T)$ is slice hyperholomorphic in the variable $s$ by Theorem \ref{Cauchy}, we can change the domain of integration to $B_r(0)\cap\cc_J$ for some $r>0$ with $ \| T \| <r$. Moreover, by hypothesis we have that $f(s)-g(s) \in (\ker{\overline{\mathcal{D}}})_{SH_L(\Omega)}$, thus by Theorem \ref{Tcost} and Proposition \ref{p2} we get
	\[
	\begin{split}
		\breve f^0(T)-\breve g^0(T)&=\frac 1{2\pi}\int_{\partial (B_r(0)\cap\cc_J)} \mathcal P^L_2(s,T)ds_J(f(s)-g(s))
		\\
		&
		=\frac 1{2\pi}\int_{\partial (B_r(0)\cap\cc_J)} \mathcal P^L_2(s,T) ds_J\alpha\\
		&= \frac 1\pi \sum_{m=1}^\infty\left(mT^{m-1}+\sum_{k=1}^m T^{m-k}\bar T^{k-1}\right) \int_{\partial (B_r(0)\cap \cc_J)}s^{-1-m} ds_J \alpha=0.
	\end{split}
	\]
	\newline
	\newline
	\fbox{$U$ is not connected}
	\newline
	\newline
	In this case we write the set $U$ in the following way
	$U=\cup_{\ell=1}^n U_\ell$ where the $U_\ell$ are the connected components of $U$. Hence, there exist constants $\alpha_\ell\in\mathbb H$ for $\ell=1,\dots, n$, such that $f(s)-g(s)=\sum_{\ell=1}^n\chi_{U_\ell}(s)\alpha_\ell$. Thus we can write
	$$
	\breve f^\circ(T)-\breve g^\circ(T)=\sum_{\ell=1}^n\frac 1{2\pi}\int_{\partial(U_\ell\cap\mathbb C_J)}\mathcal P^L_2(s,T)ds_J\alpha_\ell.
	$$
	Finally, by Lemma \ref{mono}, we get $\breve f^\circ(T)-\breve g^\circ(T)=0$.
\end{proof}

\begin{remark}
	If the set $U$ in Theorem \ref{ps2} is connected we can show the result for operators of the following form $T=T_0e_0+T_1e_1+T_2e_2+T_3e_3$. However, in order to have a well defined functional calculus also for the not connected case, as it happens for the monogenic functional calculus of McIntosh, we need to annihilate a component of the operator $T$.
\end{remark}

We conclude this section with some algebraic properties of the $P_2$-functional calculus.
\begin{proposition}
	Let $T\in\mathcal{BC}(X)$ be such that $T=T_0e_0+T_1e_1+T_2e_2$, and assume that the operators $T_\ell$, $\ell=0,1,\,2$, have real spectrum.
	\begin{itemize}
		\item If $\breve f^\circ=\overline{\mathcal{D}}f$ and $\breve f^\circ=\overline{\mathcal{D}}g$ with $f,g\in SH_L(\sigma_S(T))$ and $a\in\hh$, then
		$$ (\breve f^\circ a+\breve g^\circ)(T)=\breve f^\circ (T)a+\breve g^\circ(T). $$
		\item If $\breve f^\circ=f\overline{\mathcal{D}}$ and $\breve g^\circ=g\overline{\mathcal{D}}$ with $f,g\in SH_R(\sigma_S(T))$ and $a\in\hh$, then
		$$ (a\breve f^\circ+\breve g^\circ)(T)=a\breve f^\circ (T)+\breve g^\circ(T). $$
	\end{itemize}
\end{proposition}
\begin{proof}
	The obove identities follow immediately from the linearity of the integrals in \eqref{inte1}, resp. \eqref{inte2}.
\end{proof}

\begin{proposition}
	Let $T\in\mathcal{BC}(X)$ be such that $T=T_0e_0+T_1e_1+T_2e_2$, and assume that the operators $T_\ell$, $\ell=0,1,\,2$, have real spectrum.
	\begin{itemize}
		\item If $\breve f^{\circ}=\overline{\mathcal{D}}f$ with $f\in SH_L(\sigma_S(T))$ and assume that $f(q)=\sum_{m=0}^{\infty} q^ma_m$ with $a_m\in\mathbb H$, where this series converges on a ball $B_r(0)$ with $\sigma_S(T)\subset B_r(0)$. Then
		$$ \breve f^\circ(T)=2\sum_{m=1}^\infty\left(mT^{m-1}+\sum_{k=1}^m T^{m-k}\bar T^{k-1}\right) a_m.$$
		\item If $\breve f^{\circ}=f\mathcal{D}$ with $f\in SH_R(\sigma_S(T))$ and assume that $f(q)=\sum_{m=0}^{\infty} a_mq^m$ with $a_m\in\mathbb H$, where this series converges on a ball $B_r(0)$ with $\sigma_S(T)\subset B_r(0)$. Then
		$$ \breve f^\circ(T)=2\sum_{m=1}^\infty a_m \left(mT^{m-1}+\sum_{k=1}^m T^{m-k}\bar T^{k-1}\right).$$
	\end{itemize}
\end{proposition}
\begin{proof}
	We prove the first assertion since the second one can be proven by following similar arguments. We choose an imaginary unit $J\in\mathbb S$ and a radius $0<R<r$ such that $\sigma_S(T)\subset B_R(0)$. Then the series expansion of $f$ converges uniformly on $\partial (B_R(0)\cap\cc_J)$, and so
	$$
	\breve f^0(T)=\frac 1{2\pi} \int_{\partial (B_R(0)\cap\cc_J)} \mathcal P^L_2(s,T)\,ds_J\, \sum_{\ell=0}^\infty s^\ell a_\ell=\frac 1{2\pi}\sum_{\ell=0}^\infty \int_{\partial (B_R(0)\cap\cc_J)} \mathcal P^L_2(s,T) \, ds_J s^\ell a_\ell.
	$$
	By Proposition \ref{p2}, we further obtain
	\[
	\begin{split}
		\breve f^0(T)&=\frac 1{\pi} \int_{\partial (B_R(0)\cap\cc_J)}\sum_{m=1}^\infty \left(mT^{m-1}+\sum_{k=1}^m T^{m-k}\bar T^{k-1}\right) s^{-1-m}\,ds_J\, \sum_{\ell=0}^\infty s^\ell a_\ell\\
		& =\frac 1{\pi} \sum_{m=1}^\infty  \sum_{\ell=0}^{\infty} \left(mT^{m-1}+\sum_{k=1}^m T^{m-k}\bar T^{k-1}\right) \int_{\partial (B_R(0)\cap\cc_J)} s^{-1-m+\ell}\, ds_J\,  a_\ell\\
		&= 2\sum_{m=1}^\infty  \left(mT^{m-1}+\sum_{k=1}^m T^{m-k}\bar T^{k-1}\right) a_m.
	\end{split}
	\]
	The last equality is due to the fact that $\int_{\partial (B_R(0)\cap\cc_J)}s^{-1-m+\ell}\, ds_J$ is equal to $2\pi$ if $\ell=m$, and $0$ otherwise.
\end{proof}

\begin{theorem}
	\label{poli1}
	Let $T \in \mathcal{BC}(X)$. Let $m \in \mathbb{N}_0$, and let $U \subset \mathbb{H}$ be a bounded slice Cauchy domain with $\sigma_{S}(T) \subset U$. For every $J \in \mathbb{S}$ we have
	\begin{equation}
		\label{bege}
		P_m^2(T)=\frac{1}{2 \pi} \int_{ \partial(U \cap \mathbb{C}_J)} \mathcal P^L_2(s,T) ds_J s^{m+1},
	\end{equation}
	where $$P_m^2(T):=(m+1)T^{m}+\sum_{k=0}^m T^{m-k}\bar T^{k}.$$
\end{theorem}
\begin{proof}
	We start by considering $U$ to be the ball $B_r(0)$ with $\| T \| <r$. By Proposition \ref{p2} We know that we can expand the $\overline{\mathcal D}$-kernel operator as 
	$$
	\mathcal P^L_2(s,T)= \sum_{n=1}^{+ \infty} \left(nT^{n-1}+\sum_{k=1}^n T^{n-k}\bar T^{k-1}\right) s^{-1-n}
	$$
	for every $s \in \partial B_r(0)$. Since the series converges on $ \partial B_{r}(0)$, we have
	\[
	\begin{split}
		\frac{1}{2 \pi} \int_{ \partial(B_r(0) \cap \mathbb{C}_J)} \mathcal P^L_2(s,T) ds_J s^{m+1}&=\frac{1}{2 \pi} \sum_{n=1}^{+ \infty} \left(nT^{n-1}+\sum_{k=1}^n T^{n-k}\bar T^{k-1}\right)\int_{ \partial(B_r(0) \cap \mathbb{C}_J)} s^{-n+m} ds_J\\
		&=(m+1)T^{m}+\sum_{k=1}^{m+1} T^{m+1-k}\bar T^{k-1}\\
		&=(m+1)T^m+\sum_{k=0}^{m} T^{m-k} \overline T^{k}=P^2_m(T),
	\end{split}
	\]
	where we have used
	$$ \int_{\partial(B_{r}(0) \cap \mathbb{C}_J)} s^{-n+m}ds_J= \begin{cases}
		0 & \hbox{if} \, \, n \neq m+1\\
		2 \pi & \hbox{if} \, \, n=m+1.
	\end{cases}
	$$
	This proves the result for the case $U=B_r(0)$. Now we get the result for an arbitrary bounded Cauchy domain $U$ that contains $\sigma_S(T)$. The operator $\mathcal{P}_2^L(s,T)$ is right slice hyperholomorphic and the monomial $s^{m+1}$ is left slice hyperholomorphic on the bounded slice Cauchy domain $B_r(0) \setminus U$. By the Cauchy's integral theorem (see Theorem \ref{CIT}) we get
	
	\begin{eqnarray*}
		&&  \frac{1}{2 \pi} \int_{ \partial(B_r(0) \cap \mathbb{C}_J)} \mathcal P^L_2(s,T) ds_J s^{m+1}- \frac{1}{2 \pi} \int_{ \partial(U \cap \mathbb{C}_J)} \mathcal P^L_2(s,T) ds_J s^{m+1}\\
		&&= \frac{1}{2 \pi} \int_{ \partial\left((B_r(0)\setminus U) \cap \mathbb{C}_J\right)} \mathcal P^L_2(s,T) ds_J s^{m+1}=0.
	\end{eqnarray*}
	Finally we have
	$$ \frac{1}{2\pi} \int_{ \partial( U\cap \mathbb{C}_J)} \mathcal P^L_2(s,T) ds_J s^{m+1}=\frac{1}{2 \pi} \int_{ \partial(B_r(0) \cap \mathbb{C}_J)} \mathcal P^L_2(s,T) ds_J s^{m+1}= P^2_m(T),$$
	and this concludes the proof.
\end{proof}
Finally, by using the same methodology developed in \cite[Thm. 3.2.11]{CGKBOOK} we have the following result

\begin{lemma}
	\label{inte4}
	Let $T \in \mathcal{BC}(X)$. If $ f \in N(\sigma_S(T))$ and $U$ is a bounded slice Cauchy domain such that $ \sigma_S(T) \subset U$ and $ \overline{U} \subset dom(f)$, then we have
	$$
	\frac{1}{2\pi} \int_{\partial(U \cap \mathbb{C}_J)} \mathcal P^L_2(s,T) ds_J f(s)= \frac{1}{2\pi} \int_{\partial(U \cap \mathbb{C}_J)} f(s) ds_J \mathcal P^R_2(s,T).$$
\end{lemma}

\section{Resolvent equation and Product rule \\ for the polyanalytic functional calculus}

In the Riesz-Dunford functional calculus the main tool to show the product rule and to study the Riesz projectors is the resolvent equation. In order to recall this equation we need to introduce some notations. Let $A$ be a complex operator defined on a complex Banach space, we denote by $ R(\gamma,A):=(\gamma I-A)^{-1}$, the resolvent operator for the holomorphic functional calculus. The resolvent equation is given by
\begin{equation}
\label{hres}
R(\lambda,A) R(\mu,A)= \frac{R(\lambda,A)-R(\mu,A)}{\mu- \lambda}, \qquad \lambda, \, \mu \in \rho(A),
\end{equation}
where $ \rho(A)$ is the resolvent set of the operator $A$.
The main properties of the resolvent equation are the following
\begin{itemize}
\item The product of two different resolvent operators $R(\lambda,A) R(\mu,A)$ is transformed into the difference $R(\lambda,A)-R(\mu,A)$.
\item The difference of resolvent operators $R(\lambda,A)-R(\mu,A)$ is entangled with the Cauchy kernel $1/(\mu- \lambda)$ of the holomorphic functions.
\item The holomorphicity is maintained both in $ \lambda$ and in $ \mu \in \rho(A)$.
\end{itemize}
In this section we want to address the following problem
\begin{prob}
Is it possible to show a resolvent equation for the $P_2$-functional calculus on the $S$-spectrum that enjoys similar properties of the holomorphic resolvent equation?  
\end{prob}
In order to answer to this question it is essential the following result.
\begin{theorem}
Let $T\in \mathcal{BC}(X)$. For $q,s\in\rho_S(T)$, with $s\notin [q]$ the following equation holds
\begin{equation}\label{preseq}
\begin{split}
S^{-1}_R(s,T)\mathcal P^L_2(q,T)&+\mathcal P^R_2(s,T)S^{-1}_L(q,T)-4\mathcal Q_{c,s}(T)^{-1}\underline T \mathcal{Q}_{c,q}(T)^{-1}\\
& = [(\mathcal P^R_2(s,T)-\mathcal P_2^L(q,T))q-\bar s(\mathcal P_2^R(s,T)-\mathcal P_2^L(q,T))]\mathcal Q_s(q)^{-1},
\end{split}
\end{equation}
where $ \mathcal{Q}_{s}(q):= q^2-2s_0q+|q|^2$ and $ \underline{T}=T_1e_1+T_2e_2+T_3e_3$.
\end{theorem}
\begin{proof}
We divide the proof in nine steps. 
	
{\bf  Step I.} We multiply the $S$-resolvent equation (see \eqref{sresc}) on the right by $4\mathcal Q_{c,q}^{-1}(T)q$ and we get 
\begin{eqnarray}
\nonumber
-S^{-1}_R(s,T)F_L(q,T)q&= &[(4S^{-1}_R(s,T)\mathcal Q_{c,q}(T)^{-1}q+F_L(q,T)q)q\\
\label{rpol1}
&&-\bar s(4S^{-1}_R(s,T)\mathcal Q_{c,q}(T)^{-1}q+F_L(q,T)q)]\mathcal Q_s(q)^{-1}.
\end{eqnarray}
	
{\bf Step II.} We multiply the $S$-resolvent equation on the right by $-4T_0\mathcal Q_{c,q}^{-1}(T)$ and we obtain 
\begin{eqnarray}
	\nonumber
S^{-1}_R(s,T)T_0 F_L(q,T)&=&[(-4S^{-1}_R(s,T) T_0\mathcal Q_{c,q}(T)^{-1}-T_0 F_L(q,T))q\\
\label{rpol2}
&&-\bar s(-4S^{-1}_R(s,T) T_0\mathcal Q_{c,q}(T)^{-1}-T_0 F_L(q,T))]\mathcal Q_s(q)^{-1}.
\end{eqnarray} 
	
{\bf Step III.} We sum the equations \eqref{rpol1} and \eqref{rpol2}, we get
\begin{eqnarray}
\nonumber
S^{-1}_R(s,T)\mathcal P^L_2(q,T)&=&[-\mathcal P^L_2(q,T)q+\bar s\mathcal P^L_2(q,T)]\mathcal Q_s(q)^{-1}+4[ -S^{-1}_R(s,T)(T_0-q)\mathcal Q_{c,q}(T)^{-1}q\\ \label{rpol3}
&&+\bar s S^{-1}_R(s,T)(T_0-q)\mathcal Q_{c,q}(T)^{-1} ]\mathcal Q_s(q)^{-1}.
\end{eqnarray}

{\bf Step IV.} We multiply the $S$-resolvent equation on the left by $4\mathcal Q_{c,s}(T)^{-1}s$ and we get
\begin{eqnarray}
\nonumber
-sF_R(s,T) S^{-1}_L(q,T)&= & [(-sF_R(s,T)-4s\mathcal Q_{c,s}(T)^{-1}S^{-1}_L(q,T))q\\
\label{rpol4}
&& -\bar s(-sF_R(s,T)-4s\mathcal Q_{c,s}(T)^{-1}S^{-1}_L(q,T)) ]\mathcal Q_s(q)^{-1}.
\end{eqnarray}
{\bf Step V.} We multiply the $S$-resolvent equation on the left by $-4T_0 \mathcal Q_{c,s}(T)^{-1}$ and we get
\begin{eqnarray}
\nonumber
T_0F_R(s,T) S^{-1}_L(q,T)&=& [(T_0F_R(s,T)+4T_0\mathcal Q_{c,s}(T)^{-1}S^{-1}_L(q,T))q\\
\label{rpol5}
&&-\bar s(T_0F_R(s,T)+4T_0\mathcal Q_{c,s}(T)^{-1}S^{-1}_L(q,T)) ]\mathcal Q_s(q)^{-1}.
\end{eqnarray}
	
{\bf Step VI.} We sum the equations \eqref{rpol4} and \eqref{rpol5}, we obtain
\begin{eqnarray}
\nonumber
\mathcal P_2^R(s,T)S^{-1}_L(q,T) &=&[\mathcal P_2^R(s,T)q-\bar s\mathcal P_2^R(s,T)]\mathcal Q_s(q)^{-1}+4[\mathcal Q_{c,s}(T)^{-1}(T_0-s)S^{-1}_L(q,T)q\\ \label{eq6}
&&-\bar s\mathcal Q_{c,s}(T)^{-1}(T_0-s)S^{-1}_L(q,T)]\mathcal Q_s(q)^{-1}.
\end{eqnarray}
	
{\bf Step VII.} We sum the equations \eqref{rpol3} and \eqref{eq6}, we get
\begin{eqnarray}
\nonumber
&& S^{-1}_R(s,T)\mathcal P_2^L(q,T)+\mathcal P^R_2(s,T)S^{-1}_L(q,T)\\ \nonumber
&& =[(\mathcal P_2^R(s,T)-\mathcal P_2^L(q,T))q-\bar s(\mathcal P_2^R(s,T)-\mathcal P_2^L(q,T))]\mathcal Q_s(q)^{-1}\\
\nonumber
&&+4[(\mathcal Q_{c,s}(T)^{-1}(T_0-s)S^{-1}_L(q,T)-S^{-1}_R(s,T)(T_0-q)\mathcal Q_{c,q}(T)^{-1})q\\
\label{rpol7}
&&-\bar s(\mathcal Q_{c,s}(T)^{-1}(T_0-s)S^{-1}_L(q,T)-S^{-1}_R(s,T)(T_0-q)\mathcal Q_{c,q}(T)^{-1})]\mathcal Q_s(q)^{-1}.
\end{eqnarray}
	
{\bf Step VIII.} We manipulate the term 
\begin{eqnarray}
&&4[(\mathcal Q_{c,s}(T)^{-1}(T_0-s)S^{-1}_L(q,T)-S^{-1}_R(s,T)(T_0-q)\mathcal Q_{c,q}(T)^{-1})q\\ \label{rpol8}
&&-\bar s(\mathcal Q_{c,s}(T)^{-1}(T_0-s)S^{-1}_L(q,T)-S^{-1}_R(s,T)(T_0-q)\mathcal Q_{c,q}(T)^{-1})]\mathcal Q_s(q)^{-1}, \nonumber
\end{eqnarray}
which is in the right hand side of the equation \eqref{rpol7}. This term is the sum of the following two terms
\begin{equation}\label{s1}
\begin{split}
& 4T_0[(\mathcal Q_{c,s}(T)^{-1}S^{-1}_L(q,T)-S^{-1}_R(s,T)\mathcal Q_{c,q}(T)^{-1})q\\
& -\bar s(\mathcal Q_{c,s}(T)^{-1}S^{-1}_L(q,T)-S^{-1}_R(s,T)\mathcal Q_{c,q}(T)^{-1})]\mathcal Q_s(q)^{-1},
\end{split}
\end{equation} 
and
\begin{equation}\label{s2}
\begin{split}
& 4[(S^{-1}_R(s,T) q \mathcal Q_{c,q}(T)^{-1}-\mathcal Q_{c,s}(T)^{-1}s S^{-1}_L(q,T))q\\
& -\bar s(S^{-1}_R(s,T) q \mathcal Q_{c,q}(T)^{-1}-\mathcal Q_{c,s}(T)^{-1}s S^{-1}_L(q,T))]\mathcal Q_s(q)^{-1}.
\end{split}
\end{equation} 
Firstly, we focus on the term \eqref{s1}. By the definitions of the left and the right $S$-resolvent operators, we have
\[
\begin{split}
& \mathcal Q_{c,s}(T)^{-1}S^{-1}_L(q,T)-S^{-1}_R(s,T)\mathcal Q_{c,q}(T)^{-1}\\
& =\mathcal Q_{c,s}(T)^{-1}(q-\bar T)\mathcal Q_{c,q}(T)^{-1}-\mathcal Q_{c,s}(T)^{-1}(s-\bar T)\mathcal Q_{c,q}(T)^{-1}\\
&=\mathcal Q_{c,s}(T)^{-1}(q-s)\mathcal Q_{c,q}(T)^{-1}.
\end{split}
\]
Thus the term \eqref{s1} can be rewritten in the following way
\begin{eqnarray}
\nonumber
&& 4T_0[(\mathcal Q_{c,s}(T)^{-1}S^{-1}_L(q,T)-S^{-1}_R(s,T)\mathcal Q_{c,q}(T)^{-1})q\\ 
\nonumber
&& -\bar s(\mathcal Q_{c,s}(T)^{-1}S^{-1}_L(q,T)-S^{-1}_R(s,T)\mathcal Q_{c,q}(T)^{-1})]\mathcal Q_s(q)^{-1}\\
\nonumber
&&= 4T_0[(\mathcal Q_{c,s}(T)^{-1}(q-s)\mathcal Q_{c,q}(T)^{-1})q-\bar s(\mathcal Q_{c,s}(T)^{-1}(q-s)\mathcal Q_{c,q}(T)^{-1})]\mathcal Q_s(q)^{-1}\\
\nonumber
&&=4T_0[\mathcal Q_{c,s}(T)^{-1}(-sq+q^2+|s|^2-\bar sq)\mathcal Q_{c,q}(T)^{-1}]\mathcal Q_s(q)^{-1}=4T_0[\mathcal Q_{c,s}(T)^{-1}\mathcal Q_s(q)\mathcal Q_{c,q}(T)^{-1}]\mathcal Q_s(q)^{-1}\\
\label{4a}
&&=4T_0\mathcal Q_{c,s}(T)^{-1}\mathcal Q_{c,q}(T)^{-1}.
\end{eqnarray}
Now we focus on the term \eqref{s2}. By the definitions of the left and the right $S$-resolvent operators, we have
\[
\begin{split}
&S^{-1}_R(s,T) q \mathcal Q_{c,q}(T)^{-1}-\mathcal Q_{c,s}(T)^{-1}s S^{-1}_L(q,T)\\
&=\mathcal Q_{c,s}(T)^{-1}(s-\bar T) q \mathcal Q_{c,q}(T)^{-1}-\mathcal Q_{c,s}(T)^{-1}s (q-\bar T) \mathcal Q_{c,q}(T)^{-1}\\
&=-\mathcal Q_{c,s}(T)^{-1}(\bar T q - s\bar T) \mathcal Q_{c,q}(T)^{-1}.
\end{split}
\]
Thus the term \eqref{s2} can be rewritten in the following way
\begin{eqnarray}
\nonumber
&& 4[(S^{-1}_R(s,T) q \mathcal Q_{c,q}(T)^{-1}-\mathcal Q_{c,s}(T)^{-1}s S^{-1}_L(q,T))q\\
\nonumber
&& -\bar s(S^{-1}_R(s,T) q \mathcal Q_{c,q}(T)^{-1}-\mathcal Q_{c,s}(T)^{-1}s S^{-1}_L(q,T))]\mathcal Q_s(q)^{-1}\\
\nonumber
&&=-4[(\mathcal Q_{c,s}(T)^{-1}(\bar T q - s\bar T) \mathcal Q_{c,q}(T)^{-1})q-\bar s(\mathcal Q_{c,s}(T)^{-1}(\bar T q - s\bar T) \mathcal Q_{c,q}(T)^{-1})]\mathcal Q_s(q)^{-1}\\
\nonumber
&&=-4\mathcal Q_{c,s}(T)^{-1}(\bar Tq^2-s\bar Tq-\bar s\bar Tq+|s|^2\bar T)\mathcal Q_{c,q}(T)^{-1}\mathcal Q_s(q)^{-1}\\
\label{4b}
&& =-4\mathcal Q_{c,s}(T)^{-1}\bar T\mathcal Q_s(q)\mathcal Q_{c,q}(T)^{-1}\mathcal Q_{s}(q)^{-1}=-4\mathcal Q_{c,s}(T)^{-1}\bar T\mathcal Q_{c,q}(T)^{-1}.
\end{eqnarray}
In conclusion by \eqref{4a} and \eqref{4b} we can write 

\begin{eqnarray}
\nonumber
&4[(\mathcal Q_{c,s}(T)^{-1}(T_0-s)S^{-1}_L(q,T)-S^{-1}_R(s,T)(T_0-q)\mathcal Q_{c,q}(T)^{-1})q\\
\nonumber
&-\bar s(\mathcal Q_{c,s}(T)^{-1}(T_0-s)S^{-1}_L(q,T)-S^{-1}_R(s,T)(T_0-q)\mathcal Q_{c,q}(T)^{-1})]\mathcal Q_s(q)^{-1}\\
\label{sp1}
&=4T_0\mathcal Q_{c,s}(T)^{-1}\mathcal Q_{c,q}(T)^{-1}-4\mathcal Q_{c,s}(T)^{-1}\bar T\mathcal Q_{c,q}(T)^{-1}=4\mathcal Q_{c,s}(T)^{-1}\underline T\mathcal Q_{c,q}(T)^{-1}.
\end{eqnarray}
	
{\bf Step IX.} Finally, by \eqref{sp1} and \eqref{rpol7} we get
\[
\begin{split}
& S^{-1}_R(s,T)\mathcal P^L_2(q,T)+\mathcal P^R_2(s,T)S^{-1}_L(q,T)-4\mathcal Q_{c,s}(T)^{-1}\underline T\mathcal Q_{c,q}(T)^{-1}\\
& = [(\mathcal P_2^R(s,T)-\mathcal P_2^L(q,T))q-\bar s(\mathcal P_2^R(s,T)-\mathcal P_2^L(q,T))]\mathcal Q_{s}(q)^{-1}.
\end{split}
\] 
\end{proof}
\begin{lemma}\label{l1}
	Let $T\in\mathcal{BC}(X)$ and let $s\in\rho_S(T)$. The pseudo $s$-resolvent operator satisfies the equations 
\begin{equation}
\label{pse1}
\mathcal Q_{c,s}(T)^{-1}=\frac 14(\mathcal P^L_2(s,T)+\underline T F_L(s,T))
\end{equation}
and 
\begin{equation}
\label{pse2}
\mathcal Q_{c,s}(T)^{-1}=\frac 14(\mathcal P_2^R(s,T)+F_R(s,T)\underline T).
\end{equation}
\end{lemma}
\begin{proof}
By Theorem \ref{qf} we have
	\[
	\begin{split}
		4\mathcal Q_{c,s}(T)^{-1}&=-F_L(s,T)s+TF_L(s,T)\\
		&= -F_L(s,T)s+T_0F_L(s,T)+\underline TF_L(s,T)\\
		&=\mathcal P_2^L(s,T)+\underline T F_L(s,T).
	\end{split}
	\]
	To prove the other equality in the statement we can proceed in a similar way. 
\end{proof}
By means of the previous result we can write a resolvent equation for the $P_2$-functional calculus.
\begin{theorem}
Let $T\in \mathcal{BC}(X)$. For $q$, $s\in\rho_S(T)$, with $s\notin [q]$ the following equation holds
\begin{eqnarray}
\label{ress1}
&& \! \! \! \! \! \! \! \! \! \! \! S^{-1}_R(s,T)\mathcal P^L_2(q,T)+\mathcal P^R_2(s,T)S^{-1}_L(q,T)- \frac{1}{4} \bigl(\mathcal{P}_2^R(s,T) \underline{T} \mathcal{P}_2^L(q,T)+ \mathcal{P}_2^R(s,T) \underline{T} F_L(q,T)+\\ 
\nonumber
&&\! \! \! \! \! \! \! \! \! \! \!+ F_R(s,T) \underline{T}^2 \mathcal{P}_2^L(q,T)+ F_R(s,T) \underline{T}^3 F_L(q,T)\bigl)= [(\mathcal P_2^R(s,T)-\mathcal P_2^L(q,T))q-\bar s(\mathcal P_2^R(s,T)-\mathcal P_2^L(q,T))]\mathcal Q_{s}(q)^{-1}.
\end{eqnarray}
\end{theorem}
\begin{proof}
From the identities \eqref{pse1} and \eqref{pse2} we obtain
\begin{eqnarray*}
4^2\mathcal Q_{c,s}(T)^{-1}\underline T\mathcal Q_{c,q}(T)^{-1} &=&  \left( \mathcal{P}_2^R(s,T)+ F_R(s,T) \underline{T}\right) \underline{T} \left( \mathcal{P}_2^L(q,T)+ \underline{T} F_L(q,T)\right)\\
&=&\mathcal{P}_2^R(s,T) \underline{T} \mathcal{P}_2^L(q,T)+ \mathcal{P}_2^R(s,T) \underline{T} F_L(q,T)+\\
&&+  F_R(s,T) \underline{T}^2 \mathcal{P}_2^L(q,T) + F_R(s,T) \underline{T}^3 F_L(q,T).
\end{eqnarray*} 
Replacing this identity in \eqref{preseq} we obtain the thesis
\end{proof}

We can write the equation \eqref{ress1} in the following way
\begin{eqnarray*}
&& S^{-1}_R(s,T)\mathcal P^L_2(q,T)+\mathcal P^R_2(s,T)S^{-1}_L(q,T)- \frac{1}{4} \bigl(\mathcal{P}_2^R(s,T) \underline{T} \mathcal{P}_2^L(q,T)+ \mathcal{P}_2^R(s,T) \underline{T} F_L(q,T) \\ 
\nonumber
&& +F_R(s,T) \underline{T}^2 \mathcal{P}_2^L(q,T)+ F_R(s,T) \underline{T}^3 F_L(q,T)\bigl)=[\mathcal P_2^R(s,T)-\mathcal P_2^L(q,T)]*_{s, left} S^{-1}_L(q,s) .
\end{eqnarray*}
This equation can be considered a resolvent equation for the $P_2$-functional calculus. The main differences and the major similarities with the holomorphic resolvent equation are listed below.
\begin{itemize}
\item Due to the noncommutative setting there are two different $\mathcal{\overline{D}}$-kernel operators $ \mathcal{P}_2^L(q,T)$ and $ \mathcal{P}_2^R(s,T)$, which are right slice hyperholomorphic in $q$ and left slice hyperholomorphic in $s$, respectively.  
\item The difference $ \mathcal{P}_2^R(s,T)- \mathcal{P}_2^L(q,T)$ is suitably multiplied by the Cauchy kernel of the slice hyperholomorphic functions
\item The term 
$$ [\mathcal P_2^R(s,T)-\mathcal P_2^L(q,T)]*_{s, left} S^{-1}_L(q,s)$$
is equal not only to the product of the $ \mathcal{P}$-resolvent operators but also to other terms:
\begin{itemize}
\item the commutative version of the $S$-resolvent operators, 
\item the $F$-resolvent operators 
\end{itemize}
\item The resolvent equation preserves the slice hyperholomorphicity on the right in $s$ and on the left in $p$. 
\end{itemize} 
As it happens for the holomorphic functional calculus the resolvent equation is crucial to obtain a product formula. Before to go through this we need to recall the following technical result, see \cite{CGKBOOK}.
\begin{lemma}
	\label{app}
	Let $B \in \mathcal{B}(X)$. Let $G$ be an axially symmetric domain and assume $f \in N(G)$. Then for $p \in G$, we have
	$$ \frac{1}{2 \pi} \int_{\partial (G \cap \mathbb{C}_J)} f(s) ds_J (\bar{s}B-Bp)(p^2-2s_0p+|s|^2)^{-1}=Bf(p).$$
\end{lemma}

\begin{theorem}\label{t1}
	Let $T\in\mathcal{BC}(X)$ and assume $f\in N(\sigma_S(T))$. If $g\in SH_L(\sigma_S(T))$, then we have
	\begin{equation}\label{prl}
		\overline{\mathcal D}(fg)(T)=f(T)(\overline{\mathcal D}g)(T)+(\overline{\mathcal D}f)(T)g(T)-\mathcal D(f)(T)\underline T \mathcal D(g)(T).
	\end{equation}
	If $g\in SH_R(\sigma_S(T))$, then we have
	\begin{equation}\label{prr}
		\overline{\mathcal D}(gf)(T)=g(T)(\overline{\mathcal D}f)(T)+(\overline{\mathcal D}g)(T)f(T)-\mathcal D(g)(T)\underline T \mathcal D(f)(T).
	\end{equation}
\end{theorem}
\begin{proof}
	Let $G_1$ and $G_2$ be two bounded slice Cauchy domains such that  they contain the $S$-spectrum and $\overline G_1\subset G_2$ and $\overline G_2\subset \operatorname{dom}(f)\cap\operatorname{dom}(g)$. We choose $p\in\partial(G_1\cap\cc_J)$ and $s\in\partial(G_2\cap\cc_J)$. For every $J\in\mathbb S$, from the definitions of the $P_2$-functional calculus, of the $S$-functional calculus, of the $Q$-functional calculus and since $f$ is intrinsic by Theorem \ref{intri}, we get 
	\[
	\begin{split}
		&\overline{\mathcal D}(fg)(T)=f(T)(\overline{\mathcal D}g)(T)+(\overline{\mathcal D}f)(T)g(T)-\mathcal D(f)(T)\underline T \mathcal D(g)(T)\\
		&=\frac{1}{(2\pi)^2} \int_{\partial(G_2\cap\cc_J)}f(s)ds_J S^{-1}_R(s,T)\int_{\partial(G_1\cap\cc_J)} \mathcal P_2^L(p,T)dp_Jg(p)\\
		&+\frac{1}{(2\pi)^2} \int_{\partial(G_2\cap\cc_J)}f(s)ds_J \mathcal P_2^R(s,T) \int_{\partial(G_1\cap\cc_J)} S^{-1}_L(p,T) dp_Jg(p)\\
		&-\frac 1{\pi^2}\left( \int_{\partial(G_2\cap\cc_J)} f(s)  ds_J \mathcal Q_{c,s}(T)^{-1}\right)\underline T\left( \int_{\partial(G_1\cap\cc_J)}\mathcal Q_{c,p}(T)^{-1} dp_J g(p) \right)\\
		&=\frac 1{(2\pi)^2} \int_{\partial(G_2\cap\cc_J)}\int_{\partial(G_1\cap\cc_J)}f(s)ds_J[S^{-1}_R(s,T)\mathcal P^L_2(p,T)+\mathcal P^R_2(s,T)S^{-1}_L(p,T)\\
		&\,\,\,\,\,\,\,\,\,\,\,\,\,\,\quad\quad\quad\quad\quad\quad\quad\quad\quad\quad\quad\quad\quad-4\mathcal Q_{c,s}(T)^{-1}\underline T\mathcal Q_{c,p}(T)^{-1}] dp_Jg(p).
	\end{split}
	\]
	Now from equation \eqref{preseq} we obtain
	\[
	\begin{split}
		& \overline{\mathcal D}(fg)(T)=f(T)(\overline{\mathcal D}g)(T)+(\overline{\mathcal D}f)(T)g(T)-\mathcal D(f)(T)\underline T \mathcal D(g)(T)\\
		&=\frac 1{(2\pi)^2}\int_{\partial(G_2\cap\cc_J)}\int_{\partial(G_1\cap\cc_J)}f(s)ds_J[(\mathcal P^R_2(s,T)-\mathcal P_2^L(p,T))p\\
		&\quad\quad\quad\quad\quad\quad\quad\quad\quad\quad\quad\quad\quad\quad-\bar s(\mathcal P^R_2(s,T)-\mathcal P_2^L(p,T))]\mathcal Q_s(p)^{-1}dp_Jg(p).
	\end{split}
	\]
	Since $p\mathcal Q_s(p)^{-1}$ and $\mathcal Q_s(p)^{-1}$ are intrinsic slice hyperholomorphic on $\overline G_1$, by the Cauchy integral formula we get 
	$$
	\frac{1}{(2\pi)^2}\int_{\partial(G_2\cap\cc_J)} f(s)ds_J\mathcal P_2^R(s,T)\int_{\partial(G_1\cap\cc_J)} p\mathcal Q_{s}(p)^{-1}dp_J g(p)=0
	$$
	and
	$$
	\frac{1}{(2\pi)^2}\int_{\partial(G_2\cap\cc_J)} f(s)ds_J \bar s\mathcal P_2^R(s,T)\int_{\partial(G_1\cap\cc_J)} \mathcal Q_{s}(p)^{-1}dp_J g(p)=0.
	$$
	Therefore, we get 
	\[
	\begin{split}
		& \overline{\mathcal D}(fg)(T)=f(T)(\overline{\mathcal D}g)(T)+(\overline{\mathcal D}f)(T)g(T)-\mathcal D(f)(T)\underline T \mathcal D(g)(T)\\
		&=\frac 1{(2\pi)^2}\int_{\partial(G_2\cap\cc_J)} f(s)ds_J\int_{\partial(G_1\cap\cc_J)}[\bar s\mathcal P_2^L(p,T)-\mathcal P_2^L(p,T)p] \mathcal Q_s(p)^{-1}dp_Jg(p).
	\end{split}
	\]
	By Fubini's theorem, Lemma \ref{app} with $B:=\mathcal P^L_2(p,T)$, and the definition of the $P_2$-functional calculus we get 
	\[
	\begin{split}
		& \overline{\mathcal D}(fg)(T)=f(T)(\overline{\mathcal D}g)(T)+(\overline{\mathcal D}f)(T)g(T)-\mathcal D(f)(T)\underline T \mathcal D(g)(T)\\
		&=\frac 1{2\pi}\int_{\partial(G_1\cap\cc_J)} \mathcal P^L_2(p,T) dp_J f(p) g(p)\\
		&=\overline{\mathcal D}(fg)(T).
	\end{split}
	\]
\end{proof}

In \cite{CDPS} a product rule for the $F$- functional calculus is proved, see Theorem \ref{Pr0}. The formula is obtained in terms of the $Q$-functional calculus, i.e. the operator $ \mathcal{D}$ is involved. In the following result we show a product rule for the $F$-functional calculus in which the $P_2$-functional calculus is involved, namely the operator $ \mathcal{\overline{D}}$ plays a role.

\begin{theorem}
	Let $T\in\mathcal{BC}(X)$ and assume $f\in N(\sigma_S(T))$ and $g\in SH_L(\sigma_S(T))$. Then we have 

	\begin{eqnarray}
\label{laplca1}
		\Delta(fg)(T)=(\Delta f)(T)g(T)+f(T)(\Delta g)(T)-\frac 14 (\overline {\mathcal D} f)(T)(\overline{\mathcal D} g)(T)\\
\nonumber
		-\frac 14 (\overline{\mathcal D} f) (T) \underline T(\Delta g)(T)-\frac 14 (\Delta f)(T)\underline T (\overline{\mathcal D} g)(T)-\frac 14(\Delta f)(T) \underline T^2(\Delta g)(T).
\nonumber
	\end{eqnarray}
\end{theorem}
\begin{proof}
	Let $G_1$ and $G_2$ be two bounded slice Cauchy domains like in the proof of Theorem \ref{t1}. Let us consider $p \in \partial(G_1 \cap \mathbb{C}_J)$ and $s \in \partial(G_2 \cap \mathbb{C}_J)$. Then by the definitions of the $F$-functional calculus, of the $S$-functional calculus, of the $P_2$-functional calculus and from the fact that $f$ is intrinsic, see Theorem \ref{intri}, we get 
	\[
	\begin{split}
		& (\Delta f)(T)g(T)+f(T)(\Delta g)(T)-\frac 14 (\overline {\mathcal D} f)(T)(\overline{\mathcal D} g)(T)\\
		& -\frac 14 (\overline{\mathcal D} f) (T) \underline T(\Delta g)(T)-\frac 14 (\Delta f)(T)\underline T (\overline{\mathcal D} g)(T)-\frac 14(\Delta f)(T) \underline T^2(\Delta g)(T)\\
		&= \frac 1{(2\pi)^2}\int_{\partial(G_2\cap\cc_J)}f(s)ds_J F_R(s,T)\int_{\partial(G_1\cap\cc_J)} S^{-1}_L(p,T) dp_J g(p)\\
		&+ \frac 1{(2\pi)^2}\int_{\partial(G_2\cap\cc_J)}f(s)ds_J S^{-1}_R(s,T)\int_{\partial(G_1\cap\cc_J)} F_L(p,T) dp_J g(p)\\
		&-\frac 1{4 (2\pi)^2}\left[ \int_{\partial(G_2\cap\cc_J)}f(s)ds_J \mathcal P_2^R(s,T)\int_{\partial(G_1\cap\cc_J)} \mathcal P_2^L(p,T) dp_J g(p) \right. \\
		&+\int_{\partial(G_2\cap\cc_J)}f(s)ds_J \mathcal P_2^R(s,T) \underline T \int_{\partial(G_1\cap\cc_J)} F_L(p,T) dp_J  g(p)\\
		&+\int_{\partial(G_2\cap\cc_J)}f(s)ds_J F_R(s,T) \underline T \int_{\partial(G_1\cap\cc_J)} \mathcal P_2^L(p,T) dp_J  g(p)\\
		&\left. +\int_{\partial(G_2\cap\cc_J)}f(s)ds_J F_R(s,T) \underline T^2 \int_{\partial(G_1\cap\cc_J)} F_L(p,T) dp_J  g(p) \right]\\
		&=\frac{1}{(2\pi)^2} \int_{\partial(G_2\cap\cc_J)}\int_{\partial(G_1 \cap\cc_J)} f(s)ds_J[F_R(s,T) S^{-1}_L(p,T)+S^{-1}_R(s,T)F_L(p,T)\\
		&-\frac 14 \mathcal P^R_2(s,T)\mathcal P_2^L(p,T)-\frac 14 \mathcal P^R_2(s,T)\underline T F_L(p,T) -\frac 14 F_R(s,T) \underline T \mathcal P_2^L(p,T)\\
		& -\frac 14 F_R(s,T) \underline T^2 F_L(p,T) ]  dp_Jg(p).
	\end{split}
	\]
	From Lemma \ref{l1} we get
\begin{eqnarray*}
&& \frac{1}{4} \bigl(P^R_2(s,T)\mathcal P_2^L(p,T)+ \mathcal P^R_2(s,T)\underline T F_L(p,T) +F_R(s,T) \underline T \mathcal P_2^L(p,T)+ F_R(s,T) \underline T^2 F_L(p,T) \bigl)\\
&&= 4\mathcal Q_{c,s}(T)^{-1}\mathcal Q_{c,p}(T)^{-1}.
\end{eqnarray*}
Therefore, we get
	\[
	\begin{split}
		& (\Delta f)(T)g(T)+f(T)(\Delta g)(T)-\frac 14 (\overline {\mathcal D} f)(T)(\overline{\mathcal D} g)(T)\\
		& -\frac 14 (\overline{\mathcal D} f) (T) \underline T(\Delta g)(T)-\frac 14 (\Delta f)(T)\underline T (\overline{\mathcal D} g)(T)-\frac 14(\Delta f)(T) \underline T^2(\Delta g)(T)\\
		&=\frac 1{(2\pi)^2}\int_{\partial(G_2\cap\cc_J)}\int_{\partial(G_1 \cap\cc_J)} f(s)ds_J[F_R(s,T) S^{-1}_L(p,T)+S^{-1}_R(s,T)F_L(p,T)\\
		&-4\mathcal Q_{c,s}(T)^{-1}\mathcal Q_{c,p}(T)^{-1}]  dp_Jg(p).
	\end{split}
	\]
	Now, by \cite[Lemma $7.3.2$]{CGKBOOK}, we know that
	\[
	\begin{split}
		F_R(s,T) S^{-1}_L(p,T)+S^{-1}_R(s,T) F_L(p,T) -4\mathcal Q_{c,s}(T)^{-1}\mathcal Q_{c,p}(T)^{-1}\\
		=[(F_R(s,T)-F_L(s,T))p-\bar s(F_R(s,T)-F_L(p,T))]\mathcal Q_s(p)^{-1}.
	\end{split}
	\]
Therefore, we obtain 
	\[
	\begin{split}
		& (\Delta f)(T)g(T)+f(T)(\Delta g)(T)-\frac 14 (\overline {\mathcal D} f)(T)(\overline{\mathcal D} g)(T)\\
		& -\frac 14 (\overline{\mathcal D} f) (T) \underline T(\Delta g)(T)-\frac 14 (\Delta f)(T)\underline T (\overline{\mathcal D} g)(T)-\frac 14(\Delta f)(T) \underline T^2(\Delta g)(T)\\
		&=\frac 1{(2\pi)^2}\int_{\partial(G_2\cap\cc_J)}\int_{\partial(G_1 \cap\cc_J)} f(s)ds_J[(F_R(s,T)-F_L(s,T))p-\bar s(F_R(s,T)-F_L(p,T))] dp_Jg(p).
	\end{split}
	\]
	By using similar arguments of the proof of Theorem \ref{t1} i.e., the linearity of the integrals, the Cauchy integral formula and Lemma \ref{app}, we obtain 
	\[
	\begin{split}
		& (\Delta f)(T)g(T)+f(T)(\Delta g)(T)-\frac 14 (\overline {\mathcal D} f)(T)(\overline{\mathcal D} g)(T)\\
		& -\frac 14 (\overline{\mathcal D} f) (T) \underline T(\Delta g)(T)-\frac 14 (\Delta f)(T)\underline T (\overline{\mathcal D} g)(T)-\frac 14(\Delta f)(T) \underline T^2(\Delta g)(T)\\
		&=\frac 1{2\pi}\int_{\partial(G_1\cap\cc_J)} F_L(p,T)dp_J f(p)g(p)\\
		&=\Delta (fg)(T).
	\end{split}
	\]
\end{proof}


\section{Riesz projectors for the polyanalytic functional calculus}

In the Riesz-Dunford functional calculus the resolvent equation \eqref{hres} is used to study the Riesz projectors that are given by
$$ P_{\Omega}= \int_{\partial \Omega} (\lambda I-A)^{-1}d \lambda,$$
where $A$ is a complex operator on a complex Banach space and the set $ \Omega$ contains part of the spectrum. 
\\ The aim of this section is to investigate the Riesz projectors for the $P_2$-functional calculus. Before, we need some auxiliary results.
\begin{theorem}
	\label{fr2}
	Let $T\in\mathcal{BC}(X)$ with $s\in\rho_S(T)$ then we have 
	\begin{equation}\label{left1}
		\mathcal P_2^L(s,T)s-T\mathcal P_2^L(s,T)=4(S^{-1}_L(s,T)-\underline T\mathcal Q_{c,s}(T)^{-1})
	\end{equation}
	and
	\begin{equation}\label{right1}
		s\mathcal P^R_2(s,T)-\mathcal P^R_2(s,T)T=4(S^{-1}_R(s,T)-\mathcal Q_{c,s}(T)^{-1} \underline T).
	\end{equation}
\end{theorem}

\begin{proof}
	From the definition of the $\mathcal{\overline{D}}$-kernel operator and the following relation
	$$ F_L(s,T)s-TF_L(s,T)=-4\mathcal Q_{c,s}(T)^{-1}, $$
	we get 
	\[
	\begin{split}
		\mathcal P_2^L(s,T)s-T\mathcal P_2^L(s,T)&=(-F_L(s,T)s+T_0F_L(s,T))s-T(-F_L(s,T)s+T_0F_L(s,T))\\
		&=(-F_L(s,T)s+TF_L(s,T))s+T_0(F_L(s,T)s-TF_L(s,T))\\
		&= 4\mathcal Q_{c,s}(T)^{-1}s-4T_0\mathcal Q_{c,s}(T)^{-1}\\
		&=4(s-T_0+\underline T)\mathcal Q_{c,s}(T)^{-1}-4\underline T\mathcal Q_{c,s}(T)^{-1}\\
		&=4S_L^{-1}(s,T)-4\underline T\mathcal Q_{c,s}(T)^{-1}.
	\end{split}
	\]
	The equation \eqref{right1} follows by similar arguments.
\end{proof}
In the following result we provide a suitable generalization of the previous result.
\begin{theorem}\label{tgenreseq}
	Let $T\in\mathcal{BC}(X)$ with $s\in \rho_S(T)$ and set 
	$$ \mathcal A^L_m(s,T):=4\sum_{i=0}^{m-1} T^iS^{-1}_L(s,T)s^{m-i-1} $$
	and
	$$ \mathcal B_m^L(s,T):= 4\sum_{i=0}^{m-1}T^i\underline T\mathcal Q_{c,s}(T)^{-1} s^{m-i-1}. $$
	Then for $m\geq 1$, we have the following equation
	\begin{equation}\label{genpreseq1} 
		\mathcal P_2^L(s,T)s^m-T^m\mathcal P_2^L(s,T)=\mathcal A_m^L(s,T)-\mathcal B_m^L(s,T).
	\end{equation}
	Similarly 
	\begin{equation}\label{genpreseq2} 
		s^m \mathcal P^R_2(s,T)-\mathcal P^R_2(s,T)T^m=\mathcal A_m^R(s,T)-\mathcal B_m^R(s,T),
	\end{equation}
	where 
	$$
	\mathcal A^R_m(s,T):=4\sum_{i=0}^{m-1} s^{m-i-1} S^{-1}_R(s,T) T^i
	$$
	and 
	$$
	\mathcal B^R_m(s,T):=4\sum_{i=0}^{m-1} s^{m-i-1} \mathcal Q_{c,s}(T)^{-1} \underline T T^i.
	$$
\end{theorem}
\begin{proof}
	We will show only formula \eqref{genpreseq1} because formula \eqref{genpreseq2} follows by similar computations. We prove the result by induction on $m$. If $m=1$ we have by formula \eqref{left1}
	\[
	\begin{split}
		\mathcal P_2^L(s,T) s-T\mathcal P_2^L(s,T)&= 4(S^{-1}_L(s,T)-\underline T\mathcal Q_{c,s}(T)^{-1}) \\
		&=\mathcal A^L_1(s,T)-\mathcal B_1^L(s,T).
	\end{split}
	\]
	Now, we assume that the equation holds for $m-1$ and we will prove it for $m$. By inductive hypothesis we have  
	\begin{equation}\label{nova4}
		\begin{split}
			T^m\mathcal P_2^L(s,T)&=TT^{m-1}\mathcal P_2^L(s,T)=T(\mathcal P_2^L(s,T)s^{m-1}-\mathcal A_{m-1}^L(s,T)+\mathcal B^L_{m-1}(s,T))\\
			&=T\mathcal P_2^L(s,T)s^{m-1}-T\mathcal A_{m-1}^L(s,T)+T\mathcal B^L_{m-1}(s,T).
		\end{split}
	\end{equation}
	By using formula \eqref{left1}, we obtain
	\begin{equation}\label{nova1}
		T\mathcal P_2^L(s,T)s^{m-1}=\mathcal P_2^L(s,T)s^m-4S^{-1}_L(s,T)s^{m-1}+4\underline T\mathcal Q_{c,s}(T)^{-1}s^{m-1}.
	\end{equation}
	Moreover, we have
	\begin{equation}\label{nova2}
		\begin{split}
			T\mathcal A^L_{m-1}(s,T)&=4\sum_{i=0}^{m-2}T^{i+1}S^{-1}_L(s,T)s^{m-i-2}\\
			&=4\sum_{\ell=1}^{m-1}T^{\ell}s^{-1}_L(s,T) S^{m-\ell-1}
		\end{split}
	\end{equation}
	and
	\begin{equation}\label{nova3}
		\begin{split}
			T\mathcal B^L_{m-1}(s,T)&=4\sum_{i=0}^{m-2}T^{i+1}\underline T \mathcal Q_{c,s}(T)^{-1}s^{m-i-2}\\
			&=4\sum_{\ell=1}^{m-1}T^{\ell} \underline T\mathcal Q_{c,s}(T)^{-1}s^{m-\ell-1}.
		\end{split}
	\end{equation}
	Eventually, by inserting formulas \eqref{nova1}, \eqref{nova2} and \eqref{nova3} in \eqref{nova4}, we get 
	\[
	\begin{split}
		T^m\mathcal P_2^L(s,T)& =\mathcal P_2^L(s,T)s^m-4S^{-1}_L(s,T)s^{m-1}+4\underline T\mathcal Q_{c,s}(T)^{-1}s^{m-1}\\
		&-4\sum_{\ell=1}^{m-1}T^{\ell}S^{-1}_L(s,T)s^{m-\ell-1}+4\sum_{\ell=1}^{m-1} T^{\ell}\underline T \mathcal Q_{c,s}(T)^{-1}s^{m-\ell-1}\\
		&=\mathcal P_2^L(s,T)s^m-4\sum_{\ell=0}^{m-1} T^{\ell}S^{-1}_L(s,T) s^{m-\ell-1}+4\sum_{\ell=0}^{m-1} T^\ell\underline T\mathcal Q_{c,s}(T)^{-1}s^{m-\ell-1}\\
		&=\mathcal P_2^L(s,T)s^m-\mathcal A_m^L(s,T)+\mathcal B^L_m(s,T).
	\end{split}
	\]
\end{proof}
\begin{remark}
	Using the relation $2\underline T=T- \bar{T}$, we can also write the term $\mathcal B_m^L(s,T)$ of Theorem \ref{tgenreseq} in the following way
	$$ \mathcal B_m^L(s,T)=2 \sum_{i=1}^m T^{i+1} \mathcal{Q}_{c,s}(T)^{-1}s^{m-i-1}-2 \bar{T} \sum_{i=0}^{m-1} T^i \mathcal{Q}_{c,s}(T)^{-1} s^{m-i-1}.$$
\end{remark}

\begin{remark}
Similar equations, like the one in Theorem \ref{fr2}, have been considered for the $S$-functional calculus, $F$-functional calculus and $Q$-functional calculus, see \cite{CDPS,CSS3}:
\newline
\begin{itemize}
\item $S$ functional calculus:  \qquad $S^{-1}_L(s,T)s- TS^{-1}(s,T)= \mathcal{I}.$
\newline
\item $Q$-functional calculus: \qquad $\mathcal{Q}_{c,s}(T)^{-1}s- \bar{T} \mathcal{Q}_{c,s}(T)^{-1}=S^{-1}_L(s,T).$
\newline
\item $F$-functional calculus:\qquad $F_L(s,T)s-T F_L(s,T)= -4 \mathcal{Q}_{c,s}(T)^{-1}.$
\end{itemize}
For the sake of simplicity we have considered only the left resolvent operators.
\end{remark}  
To study the Riesz projectors we need the following result, which is possible to show in a similar way as \cite[Lemma 5.12]{CDPS}.

\begin{lemma}
	\label{harmo}
	Let $ T \in \mathcal{BC}(X)$ be such that $T=T_0+T_1e_1+T_2e_2$ and assume that the operators $T_0$, $T_1$, $T_2$ have real spectrum. Let  $G$ be a bounded Cauchy domain such that $ (\partial G) \cap \sigma_{S}(T)= \emptyset$. For every $J \in \mathbb{S}$ we have
	$$ \int_{\partial(G \cap \mathbb{C}_J)} \mathcal{Q}_{c,s}^{-1}(T) ds_J=0.$$
\end{lemma}

The interesting symmetries that appear in the equation \eqref{preseq} allow to study the Riesz projectors.

\begin{theorem}[Riesz projectors]
	\label{risz}
	Let $T=T_0+T_1e_1+T_2e_2$ and assume that the operators $T_{\ell}$, with $ \ell=0,1,2$ have real spectrum. Let $ \sigma_S(T)= \sigma_1 \cup \sigma_2$ with $ \hbox{dist}(\sigma_1, \sigma_2)>0$. Let $G_1$, $G_2 \subset \mathbb{H}$ be two bounded slice Cauchy domains such that $ \sigma_1 \subset G_1$, $ \bar{G}_1 \subset G_2$ and $\hbox{dist}(G_2, \sigma_2) >0$. Then the operator
	$$ \breve{P}_0:= \frac{1}{8 \pi} \int_{\partial(G_2 \cap \mathbb{C}_J)} s ds_J \mathcal{P}^L_2(s,T)= \frac{1}{8 \pi} \int_{\partial(G_1 \cap \mathbb{C}_J)} \mathcal{P}^R_2(p,T) dp_J p$$
	is a projection i.e.
	$$ \breve{P}_0^2= \breve{P}_0.$$
	Moreover, we have the following commutative relation with respect the operator $T$
	\begin{equation}
		\label{com}
		T \breve{P}_0= \breve{P}_0T.
	\end{equation}
\end{theorem}
\begin{proof}
	By Theorem \ref{fr2} we know that
	\begin{equation}
		\label{auxi}
		S^{-1}_R(s,T)= \frac{1}{4} \left(s \mathcal{P}^R_2(s,T)- \mathcal{P}^R_2(s,T)T\right)+ \mathcal{Q}_{c,s}(T)^{-1}.
	\end{equation}	
	
	Now, by substituting \eqref{auxi} in \eqref{preseq} we get
	\begin{eqnarray}
		\nonumber
		&&  \! \! \! \! \! \! \! \! \! \! \! \! \! \! \! \! \! \!
		\frac{1}{4}s \mathcal{P}^R_2(s,T) \mathcal{P}_2^L(p,T) - \frac{1}{4}\mathcal{P}^R_2(s,T)T\mathcal{P}_2^L(p,T)+ \mathcal{Q}_{c,s}(T)^{-1}\mathcal{P}_2^L(p,T) + \mathcal{P}_2^R(s,T) S^{-1}_L(p,T)\\
		\label{fr}
		&&\! \! \! \! \! \! \! \! \! \! \! \! \! -4 \mathcal{Q}_{c,s}(T)^{-1} \underline{T} \mathcal{Q}_{c,p}(T)= [ \left(\mathcal{P}^R_2(s,T)- \mathcal{P}_2^L(p,T)\right)p- \bar{s}\left(\mathcal{P}^R_2(s,T)- \mathcal{P}_2^L(p,T)\right)]\mathcal{Q}_{s}(p)^{-1}.
	\end{eqnarray}
	Now, we multiply the equation \eqref{fr} on the right by $p$, we get 
	\begin{eqnarray}
		\nonumber
		&&  \! \! \! \! \! \! \! \! \! \! \! \! \! \! \! \! \! \! \! \!\! \! \!\! \! \!\!\! \! \!
		\frac{1}{4}s \mathcal{P}^R_2(s,T) \mathcal{P}_2^L(p,T)p - \frac{1}{4}\mathcal{P}^R_2(s,T)T\mathcal{P}_2^L(p,T)p+ \mathcal{Q}_{c,s}(T)^{-1}\mathcal{P}_2^L(p,T)p + \mathcal{P}_2^R(s,T) S^{-1}_L(p,T)p\\
		\label{fr1}
		&&\! \! \! \! \! \! \! \! \! \! \! \! \! \! \! \! \! \! \! \!\! \! \!\! \! \!\!\! \! \! -4 \mathcal{Q}_{c,s}(T)^{-1} \underline{T} \mathcal{Q}_{c,p}(T)p= [ \left(\mathcal{P}^R_2(s,T)- \mathcal{P}_2^L(p,T)\right)p- \bar{s}\left(\mathcal{P}^R_2(s,T)- \mathcal{P}_2^L(p,T)\right)]\mathcal{Q}_{s}(p)^{-1}p.		
	\end{eqnarray}
	Now, we multiply formula \eqref{fr1} by $ds_J$ on the left and we integrate it on $ \partial(G_2 \cap \mathbb{C}_J)$ with respect to $ds_J$. Similarly, if we multiply formula \eqref{fr1} on the right by $dp_J$ and we integrate it on $ \partial(G \cap \mathbb{C}_J)$ with respect to $ dp_J$. Thus we obtain
	
	\begin{eqnarray*}
		&&   \! \! \! \! \! \! \! \! \! \!  \frac{1}{4}\int_{\partial(G_2 \cap \mathbb{C}_J)}s ds_J \mathcal{P}_{R}(s,T) \int_{\partial(G_1 \cap \mathbb{C}_J)} \mathcal{P}_{L}(p,T) dp_J p - \frac{1}{4}\int_{\partial(G_2 \cap \mathbb{C}_J)} ds_J \mathcal{P}_{R}(s,T) T \int_{\partial(G_1 \cap \mathbb{C}_J)} \mathcal{P}_{L}(p,T) dp_J p\\
		&&  \! \! \! \! \! \! \! \!\! \! - \int_{\partial(G_2 \cap \mathbb{C}_J)} ds_J \mathcal{Q}_{c,s}(T)^{-1} \int_{\partial(G_1 \cap \mathbb{C}_J)} \mathcal{P}_2^L(p,T) dp_Jp+ \int_{\partial(G_2 \cap \mathbb{C}_J)} ds_J \mathcal{P}_2^R(s,T) \int_{\partial(G_1 \cap \mathbb{C}_J)} S^{-1}_L(p,T) dp_Jp+\\
		&& \! \! \! \! \! \! \! \!\! \! -4\int_{\partial(G_2 \cap \mathbb{C}_J)} ds_J \mathcal{Q}_{c,s}(T)^{-1} \underline{T}\int_{\partial(G_1 \cap \mathbb{C}_J)} \mathcal{Q}_{c,p}(T)^{-1} dp_J = \int_{\partial(G_2 \cap \mathbb{C}_J)} ds_J \int_{\partial(G_1 \cap \mathbb{C}_J)} [ \left(\mathcal{P}^R_2(s,T)- \mathcal{P}_2^L(p,T)\right)p\\
		&&- \bar{s}\left(\mathcal{P}^R_2(s,T)- \mathcal{P}_2^L(p,T)\right)]\mathcal{Q}_{s}(p)^{-1} dp_Jp
	\end{eqnarray*}
	By Lemma \ref{mono} and Lemma \ref{harmo} we have
	\begin{eqnarray}
		\label{auxi2}
		&& 4 (2 \pi)^2 \breve{P}_0^2 = \int_{\partial(G_2 \cap \mathbb{C}_J)} ds_J \int_{\partial(G_1 \cap \mathbb{C}_J)} [ \left(\mathcal{P}^R_2(s,T) -\mathcal{P}_2^L(p,T)\right)p\\ 
		\nonumber
		&& -\bar{s}\left(\mathcal{P}^R_2(s,T)- \mathcal{P}_2^L(p,T)\right)]\mathcal{Q}_{s}(p)^{-1} dp_Jp.
	\end{eqnarray}
	Now, since the functions $ p \mapsto \mathcal{Q}_{s}(p)^{-1}$ and $ p \mapsto \mathcal{Q}_s(p)^{-1}$ are slice hyperholomorphic and do not have singularities inside $ \partial(G_1 \cap \mathbb{C}_J)$ by the Cauchy theorem we get
\begin{equation}
\label{Cauchy9}
\int_{\partial(G_1 \cap \mathbb{C}_J)} p \mathcal{Q}_{s}(p)^{-1} dp_J p^2= \int_{\partial(G_1 \cap \mathbb{C}_J)} \mathcal{Q}_s(p) dp_J p=0.
\end{equation}
	This implies that formula \eqref{auxi2} can be written as
	\begin{eqnarray*}
		(\breve{P}_0)^2&=& - \frac{1}{4(2 \pi)^2} \int_{\partial(G_2 \cap \mathbb{C}_J)} ds_J \int_{\partial(G_1 \cap \mathbb{C}_J)} \mathcal{P}^L_2(p,T) \mathcal{Q}_s(p)^{-1} pdp_J p\\
		&& + \frac{1}{4 (2\pi)^2} \int_{\partial(G_1 \cap \mathbb{C}_J)} ds_J \int_{\partial(G_1 \cap \mathbb{C}_J)} \bar{s} \mathcal{P}^L_2(p,T) \mathcal{Q}_s(p)^{-1} p dp_J\\
		&=& \frac{1}{4(2 \pi)^2}\int_{\partial(G_2 \cap \mathbb{C}_J)}\int_{\partial(G_1 \cap \mathbb{C}_J)} ds_J [\bar{s} \mathcal{P}^L_2(p,T)-\mathcal{P}^L_2(p,T)p] \mathcal{Q}_{s}(p)^{-1}  dp_Jp.
	\end{eqnarray*}
	By Fubini's theorem and Lemma \ref{app} with $B:=\mathcal{P}^L_2(p,T)$ we get
	$$
	\breve{P}_{0}^2= \frac{1}{2 \pi} \int_{\partial(G_1 \cap \mathbb{C}_J)}  \mathcal{P}^L_2(p,T) dp_Jp=\breve{P}_{0}.
	$$
	\newline
	Now, we want to show the commutativity relation \eqref{com}. By \eqref{left1} we know that
	$$ T \mathcal{P}^L_2(p,T)= \mathcal{P}^L_2(s,T)s -4 \left( S^{-1}_L(s,T)- \underline{T} \mathcal{Q}_{c,s}(T)^{-1}\right).$$
	From the definition of Riesz projector we get
$$
T \breve{P}_0= \frac{1}{8 \pi}\int_{\partial(G_1 \cap \mathbb{C}_J)} \mathcal{P}^L_2(s,T) ds_J s^2- \frac{1}{2 \pi}\int_{\partial(G_1 \cap \mathbb{C}_J)} S^{-1}_L(s,T) ds_J s+ \underline{T} \int_{\partial(G_1 \cap \mathbb{C}_J)} \mathcal{Q}_{c,s}(T)^{-1} ds_J s. $$
	On the other side, by \eqref{right1} 
	we obtain
	$$ \mathcal{P}^R_2(s,T)T= s \mathcal{P}^R_2(s,T)-4 \left( S^{-1}_R(s,T)- \underline{T} \mathcal{Q}_{c,s}(T)^{-1}\right).$$
	This togetehr with the definition of Riesz projectors we get
$$ \breve{P}_0 T= \frac{1}{8 \pi} \int_{\partial(G_1 \cap \mathbb{C}_J)} s^2 ds_J \mathcal{P}^R_2(s,T)- \frac{1}{2 \pi} \int_{\partial(G_1 \cap \mathbb{C}_J)} s ds_J S^{-1}_R(s,T)+ \frac{\underline{T}}{2 \pi} \int_{G_1 \cap \mathbb{C}_J} s ds_J \mathcal{Q}_{c,s}(T)^{-1}.$$
Finally, by Theorem \ref{intri} we have the statement.
\end{proof}

\begin{remark}
	Another way to study the Riesz projectors is to rewrite the projector $\breve{P}_0$ of Theorem \ref{risz} in another way. Let us consider the set $G_1$ as in the hypothesis of Theorem \ref{risz}, then by \cite[Lemma 7.4.1]{CGKBOOK} we know that
	$$ \int_{\partial(G_1 \cap \mathbb{C}_J)} F_L(s,T) ds_J s=0.$$
	This implies that
	$$ \breve{P}_0 = - \frac{1}{8 \pi}\int_{\partial(G_1 \cap \mathbb{C}_J)} F_L(s,T) ds_J s^2- \frac{T_0}{8 \pi} \int_{\partial(G_1 \cap \mathbb{C}_J)}F_L(s,T) ds_J s=- \frac{1}{8 \pi}\int_{\partial(G_1 \cap \mathbb{C}_J)} F_L(s,T) ds_J s^2:= \breve{P}.$$
	The operator $ \breve{P}$ is the Riesz projector for the $F$-functional calculus, see \cite[Thm. 7.4.2]{CGKBOOK}, namely $\breve{P}^2= \breve{P}$. Therefore, also the operator $ \breve{P}_0$ is a projector.
\end{remark}
\section{A new resolvent equation for the $Q$-functional calculus}
In this section we provide a new resolvent equation for the $Q$-functional calculus. In \cite{CDPS} the following equation is proved

\begin{eqnarray}\label{wrong}
	&s\mathcal{Q}_{c,s}(T)^{-1} \mathcal{Q}_{c,p}(T)^{-1} p-s\mathcal{Q}_{c,s}(T)^{-1} \overline T\mathcal{Q}_{c,p}(T)^{-1}\nonumber\\
	& -\mathcal{Q}_{c,s}(T)^{-1} \overline T\mathcal{Q}_{c,p}(T)^{-1} p+\mathcal{Q}_{c,s}(T)^{-1}\overline T^2\mathcal{Q}_{c,p}(T)^{-1}\nonumber\\
	&= \left[(s\mathcal{Q}_{c,s}(T)^{-1}-p\mathcal{Q}_{c,p}(T)^{-1})p-\overline s(s\mathcal{Q}_{c,s}(T)^{-1}-p\mathcal{Q}_{c,p}(T)^{-1})\right](p^2-2s_0p+|s|^2)^{-1}\\
	&+\left[(\overline T\mathcal{Q}_{c,p}(T)^{-1}-\mathcal{Q}_{c,s}(T)^{-1}\overline T)p-\overline s(\overline T\mathcal{Q}_{c,p}(T)^{-1}-\mathcal{Q}_{c,s}(T)^{-1}\overline T)\right](p^2-2s_0p+|s|^2)^{-1}.\nonumber
\end{eqnarray}
This equation can be considered as a resolvent equation for the $Q$-functional calculus because the left slice hyperholomorphicity is maintained in the variables $s$ as well as the right slice hyperholomorphicity is preserved in $p$. Furthermore in \cite{CDPS}, this equation is used to study the Riesz projectors for the $Q$-functional calculus.
However, in formula \eqref{wrong} the term $ \mathcal{Q}_{c,s}(T)^{-1}- \mathcal{Q}_{c,p}(T)^{-1}$ is not transformed into the product of $ \mathcal{Q}_{c,s}(T)^{-1}$ and $ \mathcal{Q}_{c,p}(T)^{-1}$. Therefore, one of the main properties of the resolvent equation is missing.
\\ The goal of this section is to obtain a resolvent equation for the $Q$-functional calculus in which a term of the following form
$$[\mathcal{Q}_{c,s}(T)^{-1}- \mathcal{Q}_{c,p}(T)^{-1}]*_{s, left} S^{-1}_L(s,q)$$
is transformed into the product of $\mathcal{Q}_{c,s}(T)^{-1}$ and $\mathcal{Q}_{c,p}(T)^{-1}$ and other terms involving the resolvent operators of the commutative $S$-functional calculus. Moreover we want to maintain the slice hyperholomorphicity
\\ In order to achieve this aim we need to recall a suitable modification of the classic $S$-resolvent equation, see \cite[Thm. 6.7]{ACDS2}.
\begin{theorem}
	Let $T \in \mathcal{BC}(X)$ and $B \in \mathcal{B}(X)$ such that it commutes with $T$, then we have
	\begin{eqnarray}
		\label{Bres}
		S^{-1}_R(s,T)B S^{-1}_L(p,T)&=& [\left(S^{-1}_R(s,T)B-BS^{-1}_L(p,T)\right)p+\\
		\nonumber
		&&- \bar{s}\left(S^{-1}_R(s,T)B-BS^{-1}_L(p,T)\right)] \mathcal{Q}_s(p)^{-1},
	\end{eqnarray}
	where $ \mathcal{Q}_s(p):= p^2-2s_0p+|s|^2$.
\end{theorem}
\begin{remark}
	If we consider $B = \mathcal{I}$ in \eqref{Bres} we get \eqref{sresc}.
\end{remark}
Moreover, in order to obtain a new resolvent equation for the $Q$-functional calculus, it is crucial to write the pseudo $S$-resolvent operator in terms of the $F$-resolvent operators, see Theorem \ref{qf}.

Now, we have all the tools to obtain a new resolvent equation for the $Q$-functional calculus
\begin{theorem}
	\label{Hres}
	Let $T \in \mathcal{BC}(X)$. For $s$, $p \in \rho_S(T)$ with $s \notin [p]$ we have the following equation
	\begin{eqnarray}
		\label{qres}
		&& \mathcal{Q}_{c,s}(T)^{-1} S^{-1}_L(p,T)+ S^{-1}_R(s,T) \mathcal{Q}_{c,p}(T)^{-1}-2 \mathcal{Q}_{c,s}(T)^{-1} \underline{T} \mathcal{Q}_{c,p}(T)^{-1}=\\
		\nonumber
		&&= [(\mathcal{Q}_{c,s}(T)^{-1}- \mathcal{Q}_{c,p}(T)^{-1})p- \bar{s}(\mathcal{Q}_{c,s}(T)^{-1}- \mathcal{Q}_{c,p}(T)^{-1})] \mathcal{Q}_s(p)^{-1},
	\end{eqnarray}
	where $ \mathcal{Q}_s(p):= p^2-2s_0p+|s|^2$ and $ \underline{T}=T_1e_1+T_2e_2+T_3e_3$.
\end{theorem}
\begin{proof}
	We will show this result in seven steps.
	\newline
	\newline
	{\bf Step I.} We consider $B=T$ in \eqref{Bres} and we multiply it on the right by $4 \mathcal{Q}_{c,p}(T)^{-1}$, then we get
	\begin{eqnarray}
		\label{one}
		-S^{-1}_R(s,T) T F_L(p,T)&=& [\left(4 S^{-1}_R(s,T) T \mathcal{Q}_{c,p}(T)^{-1}+T F_L(p,T)\right)p+\\
		\nonumber
		&&- \bar{s}\left(4 S^{-1}_R(s,T) T \mathcal{Q}_{c,p}(T)^{-1}+T F_L(p,T)\right)] \mathcal{Q}_s(p)^{-1}.
	\end{eqnarray}
	{\bf Step II.} We consider $B= \mathcal{I}$ in \eqref{Bres} and we multiply it on the right by $-4 \mathcal{Q}_{c,p}(T)^{-1}p$, then we obtain
	\begin{eqnarray}
		\label{second}
		S^{-1}_R(s,T) F_L(s,T)p&=& [\left(-4 S^{-1}_R(s,T) \mathcal{Q}_{c,p}(T)p- F_L(p,T)p\right)p+\\
		\nonumber
		&&- \bar{s}\left(-4 S^{-1}_R(s,T) \mathcal{Q}_{c,p}(T)p- F_L(p,T)p\right)] \mathcal{Q}_s(p)^{-1}.
	\end{eqnarray}
	{\bf Step III.} We substitute $B=T$ in \eqref{Bres} and we multiply it on the left by $4 \mathcal{Q}_{c,s}(T)^{-1}$, then we get
	\begin{eqnarray}
		\label{third}
		- F_R(s,T) T S^{-1}_L(s,T)&=&[\left(- F_R(s,T)T-4 \mathcal{Q}_{c,s}(T)^{-1} T S^{-1}_L(p,T)\right)p+\\
		\nonumber
		&& - \bar{s}\left(- F_R(s,T)T-4 \mathcal{Q}_{c,s}(T)^{-1} T S^{-1}_L(p,T)\right)] \mathcal{Q}_{s}(p)^{-1}.
	\end{eqnarray}
	{\bf Step IV.} We substitute $B= \mathcal{I}$ in \eqref{Bres} and we multiply it on the left by $-4s \mathcal{Q}_{c,s}(T)^{-1}$, then we obtain
	\begin{eqnarray}
		\label{fourth}
		s F_R(s,T) S^{-1}_L(s,T)&=&[ \left(s F_R(s,T)+4 s \mathcal{Q}_{c,s}(T)^{-1} S^{-1}_L(p,T)\right)p+\\
		\nonumber
		&& - \bar{s}\left(s F_R(s,T)+4 s \mathcal{Q}_{c,s}(T)^{-1} S^{-1}_L(p,T)\right) ] \mathcal{Q}_{s}(p)^{-1}.
	\end{eqnarray}  
	{\bf Step V.} We make the sum of formulas \eqref{one}, \eqref{second}, \eqref{third}, \eqref{fourth} and by Theorem \ref{qf} we get
	\begin{eqnarray}
		\label{seventh}
		&& -4 S^{-1}_R(s,T) \mathcal{Q}_{c,p}(T)^{-1}-4 \mathcal{Q}_{c,s}(T)^{-1} S^{-1}_L(s,T) \\
		\nonumber
		&& = [ \left(4 \mathcal{Q}_{c,p}(T)^{-1}-4 \mathcal{Q}_{c,s}(T)^{-1}\right)p- \bar{s}\left(4 \mathcal{Q}_{c,p}(T)^{-1}-4 \mathcal{Q}_{c,s}(T)^{-1}\right)] \mathcal{Q}_{s}(p)^{-1}+\\
		\nonumber
		&& +4 \bigl[ \bigl(S^{-1}_R(s,T)T \mathcal{Q}_{c,p}(T)^{-1}-S^{-1}_R(s,T) \mathcal{Q}_{c,p}(T)^{-1}p- \mathcal{Q}_{c,s}(T)^{-1} T S^{-1}_L(p,T)+\\
		\nonumber
		&& +s \mathcal{Q}_{c,s}(T)^{-1} S^{-1}_L(p,T)\bigl)p- \bar{s}\bigl(S^{-1}_R(s,T)T \mathcal{Q}_{c,p}(T)^{-1}-S^{-1}_R(s,T) \mathcal{Q}_{c,p}(T)^{-1}p+\\
		\nonumber
		&& - \mathcal{Q}_{c,s}(T)^{-1} T S^{-1}_L(p,T)+s \mathcal{Q}_{c,s}(T)^{-1} S^{-1}_L(p,T) \bigl) \bigl] \mathcal{Q}_{s}(p)^{-1}.
	\end{eqnarray}
	{\bf Step VI.} We show that
	
	\begin{eqnarray}
		\label{sixth}
		&& \bigl[ \bigl(S^{-1}_R(s,T)T \mathcal{Q}_{c,p}(T)^{-1}- \mathcal{Q}_{c,s}(T)^{-1} T S^{-1}_L(p,T)\bigl)p+\\
		\nonumber
		&& - \bar{s}\bigl(S^{-1}_R(s,T)T \mathcal{Q}_{c,p}(T)^{-1} - \mathcal{Q}_{c,s}(T)^{-1} T S^{-1}_L(p,T) \bigl] \mathcal{Q}_{s}(p)^{-1}= - \mathcal{Q}_{c,s}(T)^{-1} T \mathcal{Q}_{c,p}(T)^{-1}.
	\end{eqnarray}
	
	We focus on proving formula \eqref{sixth}. First of all, we observe that by the definition of the $S$-resolvent operators we have
	\begin{eqnarray*}
		&& S^{-1}_{R}(s,T) T \mathcal{Q}_{c,p}(T)^{-1}- \mathcal{Q}_{c,s}(T)^{-1}T S^{-1}_L(p,T)\\
		&& = \mathcal{Q}_{c,s}(T)^{-1}(s \mathcal{I}- \bar{T})T \mathcal{Q}_{c,p}(T)^{-1}- \mathcal{Q}_{c,s}(T)^{-1}T(p \mathcal{I}- \bar{T}) \mathcal{Q}_{c,p}(T)^{-1}\\
		&& = \mathcal{Q}_{c,s}(T)^{-1} sT \mathcal{Q}_{c,p}(T)^{-1}- \mathcal{Q}_{c,s}(T)^{-1} Tp \mathcal{Q}_{c,p}(T)^{-1}.
	\end{eqnarray*}
	Thus we get
	\begin{eqnarray*}
		&& \bigl[ \bigl(S^{-1}_R(s,T)T \mathcal{Q}_{c,p}(T)^{-1}- \mathcal{Q}_{c,s}(T)^{-1} T S^{-1}_L(p,T)\bigl)p+\\
		&&  \, \, \, \, - \bar{s}\bigl(S^{-1}_R(s,T)T \mathcal{Q}_{c,p}(T)^{-1} - \mathcal{Q}_{c,s}(T)^{-1} T S^{-1}_L(p,T) \bigl] \mathcal{Q}_{s}(p)^{-1}\\
		&& =  \biggl(\mathcal{Q}_{c,s}(T)^{-1}sTp \mathcal{Q}_{c,p}(T)^{-1}- \mathcal{Q}_{c,s}(T)^{-1} T p^2 \mathcal{Q}_{c,p}(T)^{-1}+\\
		&& \, \, \, \, -\mathcal{Q}_{c,s}(T)^{-1} |s|^2 T \mathcal{Q}_{c,p}(T)^{-1}+ \mathcal{Q}_{c,s}(T)^{-1} \bar{s} Tp \mathcal{Q}_{c,p}(T)^{-1}\biggl) \mathcal{Q}_{s}(p)^{-1} \\
		&& = \biggl( \mathcal{Q}_{c,s}(T)^{-1}(s+ \bar{s}) Tp \mathcal{Q}_{c,p}(T)^{-1}- \mathcal{Q}_{c,s}(T)^{-1} T p^2 \mathcal{Q}_{c,p}(T)^{-1}
		- \mathcal{Q}_{c,s}(T)^{-1} |s|^2 T \mathcal{Q}_{c,p}(T)^{-1} \biggl) \mathcal{Q}_{s}(p)^{-1}\\
		&& = - \mathcal{Q}_{c,s}(T)^{-1} T \mathcal{Q}_{s}(p)\mathcal{Q}_{c,p}(T)^{-1} \mathcal{Q}_{s}(p)^{-1}\\
		&& = - \mathcal{Q}_{c,s}(T)^{-1} T \mathcal{Q}_{c,p}(T)^{-1}.
	\end{eqnarray*}
	{\bf Step VII.} The following equality follows by \eqref{4b}
	\begin{eqnarray}
		\label{fiveth}
		&&  \bigl[\bigl(s \mathcal{Q}_{c,s}(T)^{-1} S^{-1}_L(p,T) -S^{-1}_R(s,T) \mathcal{Q}_{c,p}(T)^{-1}p\bigl)p+\\
		\nonumber
		&& - \bar{s}\bigl(s \mathcal{Q}_{c,s}(T)^{-1} S^{-1}_L(p,T)-S^{-1}_R(s,T) \mathcal{Q}_{c,p}(T)^{-1}p \bigl)\bigl] \mathcal{Q}_{s}(p)^{-1}= \mathcal{Q}_{c,s}(T)^{-1} \bar{T} \mathcal{Q}_{c,p}(T)^{-1}.
	\end{eqnarray}
	\newline
	\newline
	{\bf Step VIII.} We put together \eqref{sixth} and \eqref{fiveth} to obtain
	\begin{eqnarray}
		\label{eight}
		&& \bigl[ \bigl(S^{-1}_R(s,T)T \mathcal{Q}_{c,p}(T)^{-1}-S^{-1}_R(s,T) \mathcal{Q}_{c,p}(T)^{-1}p- \mathcal{Q}_{c,s}(T)^{-1} T S^{-1}_L(p,T)+\\
		\nonumber
		&&  \, \, \, \, \,+s \mathcal{Q}_{c,s}(T)^{-1} S^{-1}_L(p,T)\bigl)p- \bar{s}\bigl(S^{-1}_R(s,T)T \mathcal{Q}_{c,p}(T)^{-1}-S^{-1}_R(s,T) \mathcal{Q}_{c,p}(T)^{-1}p+\\
		\nonumber
		&& \, \, \, \, \, - \mathcal{Q}_{c,s}(T)^{-1} T S^{-1}_L(p,T)+s \mathcal{Q}_{c,s}(T)^{-1} S^{-1}_L(p,T) \bigl) \bigl] \mathcal{Q}_{s}(p)^{-1}\\
		\nonumber
		&& = - \mathcal{Q}_{c,s}(T)^{-1}T \mathcal{Q}_{c,p}(T)^{-1}+ \mathcal{Q}_{c,s}(T) \bar{T} \mathcal{Q}_{c,p}(T)^{-1}\\
		\nonumber
		&& = -2 \mathcal{Q}_{c,s}(T)^{-1} \underline{T} \mathcal{Q}_{c,p}(T)^{-1}.
	\end{eqnarray}
	Finally by putting formula \eqref{eight} in \eqref{seventh} we get
	\begin{eqnarray*}
		&& S^{-1}_R(s,T) \mathcal{Q}_{c,p}(T)^{-1}+ \mathcal{Q}_{c,s}(T)^{-1} S^{-1}_L(p,T) =\\
		&& [\left( \mathcal{Q}_{c,s}(T)^{-1}- \mathcal{Q}_{c,p}(T)^{-1}\right)p- \bar{s}\left( \mathcal{Q}_{c,s}(T)^{-1}- \mathcal{Q}_{c,p}(T)^{-1}\right)] \mathcal{Q}_{s}(p)^{-1}+2 \mathcal{Q}_{c,s}(T)^{-1} \underline{T} \mathcal{Q}_{c,p}(T)^{-1}.
	\end{eqnarray*}
	This proves the statement.
\end{proof}
By using the resolvent equation \eqref{wrong}, in \cite{CDPS} we show the following product rule for the $Q$-functional calculus
\begin{eqnarray}
	\label{product1}
	2[\mathcal{D} \left((.) fg\right)(T)- \overline{T}\mathcal{D}(fg)(T)]&=&f(T) \mathcal{D} \left((.)g\right)(T)-f(T) \overline{T} \mathcal{D}(g)(T)+\\
	\nonumber
	&&+ \mathcal{D} \left(f(.)\right)(T) g(T)- \mathcal{D}(f)(T) \overline{T} g(T).
\end{eqnarray}
However, by using formula \eqref{qres}, it is possible to obtain a more interesting and nice formula for the product rule of the $Q$-functional calculus.

\begin{theorem}
	\label{Pr4}
	Let $T \in \mathcal{BC}(X)$. We assume that $f \in N(\sigma_S(T))$ and $ g \in SH_L(\sigma_S(T))$ then we have
	$$ \mathcal{D}(fg)(T)=f(T) (\mathcal{D}g)(T)+ (\mathcal{D}f)(T) g(T)+ (\mathcal{D}f)(T)  \underline{T} (\mathcal{D}g)(T).$$
	If $g \in SH_R(\sigma_S(T))$ then we have
\begin{equation}
\label{extraright}
\mathcal{D}(gf)(T)=f(T) (\mathcal{D}g)(T)+ (\mathcal{D}f)(T) g(T)+ (\mathcal{D}f)(T)  \underline{T} (\mathcal{D}g)(T).
\end{equation}
\end{theorem}
\begin{proof}
	Let $G_1$ and $G_2$ be two bounded slice Cauchy domains as in the proof of Theorem \ref{t1}  
	. Let us consider $P\in\partial(G_1\cap\cc_J)$ and $s\in\partial(G_1\cap\cc_J)$. By the definitions of the $S$-functional calculus and the $Q$-functional calculus we get
	\begin{eqnarray*}
		&& f(T) (\mathcal{D}g)(T)+ (\mathcal{D}f)(T) g(T)+ (\mathcal{D}f)(T)  \underline{T} (\mathcal{D}g)(T)\\
		&& =- \frac{1}{2 \pi^2} \int_{\partial(G_2 \cap \mathbb{C}_J)} S^{-1}_L(s,T) ds_J f(s)\int_{\partial(G_1 \cap \mathbb{C}_J)} \mathcal{Q}_{c,p}(T)^{-1} dp_J  g(p)\\
		&&- \frac{1}{2 \pi^2} \int_{\partial (G_2 \cap \mathbb{C}_J)} \mathcal{Q}_{c,s}(T)^{-1} ds_J f(s) \int_{\partial (G_1 \cap \mathbb{C}_J)} S^{-1}_L(p,T) dp_J g(p)\\
		&& + \frac{1}{\pi^2} \int_{\partial(G_2 \cap \mathbb{C}_J)} \mathcal{Q}_{c,s}(T)^{-1} ds_J f(s) \underline{T} \int_{\partial(G_1 \cap \mathbb{C}_J)} \mathcal{Q}_{c,p}(T)^{-1} dp_J g(p).
	\end{eqnarray*}
	Since by hypothesis the function $f$ is intrinsic, by Theorem \ref{intri} and by Theorem \ref{Hres} we get
	\begin{eqnarray*}
		&& f(T) (\mathcal{D}g)(T)+ (\mathcal{D}f)(T) g(T)+ (\mathcal{D}f)(T)  \underline{T} (\mathcal{D}g)(T)\\
		&& =\frac{1}{2 \pi^2} \int_{\partial(G_2 \cap \mathbb{C}_J)} \int_{\partial(G_1 \cap \mathbb{C}_J)} f(s) ds_J [- S^{-1}_R(s,T) \mathcal{Q}_{c,p}(T)^{-1}- \mathcal{Q}_{c,s}(T)^{-1} S^{-1}_L(p,T)+\\
		&& \quad +2 \mathcal{Q}_{c,s} (T)^{-1} \underline{T} \mathcal{Q}_{c,p}(T)^{-1}] dp_J g(p)\\
		&&=- \frac{1}{2 \pi^2} \int_{\partial(G_2 \cap \mathbb{C}_J)} \int_{\partial( G_1 \cap \mathbb{C}_J)} f(s) ds_J [(\mathcal{Q}_{c,s}(T)^{-1}- \mathcal{Q}_{c,p}(T)^{-1})p- \bar{s}(\mathcal{Q}_{c,s}(T)^{-1}- \mathcal{Q}_{c,p}(T)^{-1})] \cdot \\
		&& \quad \cdot \mathcal{Q}_{s}(p)^{-1} dp_J g(p)\\
		&&= - \frac{1}{2\pi^2} \int_{\partial(G_2 \cap \mathbb{C}_J)} f(s) ds_J \int_{\partial(G_1 \cap \mathbb{C}_J)} \mathcal{Q}_{c,s}(T)^{-1} p \mathcal{Q}_{s}(p)^{-1} dp_J g(p)+\\
		&& \quad + \frac{1}{2\pi^2} \int_{\partial(G_2 \cap \mathbb{C}_J)} f(s) ds_J \int_{\partial(G_1 \cap \mathbb{C}_J)} \mathcal{Q}_{c,p}(T)^{-1} p \mathcal{Q}_{s}(p)^{-1} dp_J g(p)\\
		&& \quad + \frac{1}{2\pi^2} \int_{\partial(G_2 \cap \mathbb{C}_J)} f(s) ds_J \int_{\partial(G_1 \cap \mathbb{C}_J)} \bar{s}\mathcal{Q}_{c,s}(T)^{-1} \mathcal{Q}_{s}(p)^{-1} dp_J g(p)\\
		&& \quad - \frac{1}{2\pi^2} \int_{\partial(G_2 \cap \mathbb{C}_J)} f(s) ds_J \int_{\partial(G_1 \cap \mathbb{C}_J)}  \bar{s}\mathcal{Q}_{c,p}(T)^{-1}  \mathcal{Q}_{s}(p)^{-1} dp_J g(p).
	\end{eqnarray*}
	Since the maps $p \mapsto p \mathcal{Q}_{s}(p)^{-1}$ and $ p \mapsto \mathcal{Q}_{s}(p)$ are intrinsic slice hyperholomorphic on $ \bar{G}_1$, by the Cauchy integral formula we get that the first and the third integrals in the above formula are zero. By Lemma \ref{app} 
	and the definition of $Q$-functional calculus we get
	\begin{eqnarray*}
		&& f(T) (\mathcal{D}g)(T)+ (\mathcal{D}f)(T) g(T)+ (\mathcal{D}f)(T)  \underline{T} (\mathcal{D}g)(T)=\\
		&&= - \frac{1}{2 \pi^2} \int_{\partial(G_2 \cap \mathbb{C}_J)} f(s) ds_J \int_{\partial(G_1 \cap \mathbb{C}_J)} [\bar{s} \mathcal{Q}_{c,p}(T)^{-1}- \mathcal{Q}_{c,p}(T)^{-1}p] \mathcal{Q}_s(p) \,dp_J\, g(p)\\
		&&= - \frac{1}{ \pi} \int_{\partial(G_1 \cap \mathbb{C}_J) } \mathcal{Q}_{c,p}(T)^{-1} dp_J f(p) g(p)\\
		&& = - \frac{1}{ \pi} \int_{\partial(G_1 \cap \mathbb{C}_J) } \mathcal{Q}_{c,p}(T)^{-1} dp_J (fg)(p)\\
		&& =\mathcal{D}(fg)(T).
	\end{eqnarray*}
Formula \eqref{extraright} follows by similar computations.
\end{proof}

Now, we show an application of Theorem \ref{Pr4}.
\begin{lemma}\label{delta}
	Let $n \geq 1$. Then we have
	$$ \Delta(q^{n+1})=(\Delta q^n) q_0- \mathcal{\overline{D}}q^n.$$
\end{lemma}
\begin{proof}
	By replacing the operator $T \in \mathcal{BC}(X)$ with a generic quaternion $q \in \mathbb{H}$ in Theorem \ref{Pr4} we get
	\begin{eqnarray*}
		\mathcal{D}(q^{n+1})&=& q \mathcal{D}(q^n)+ \mathcal{D}(q) q^n+ (\mathcal{D}q) \underline{q} (\mathcal{D} q^n)\\
		&=& q \mathcal{D}(q^n)-2 q^n-2 \underline{q} \mathcal{D}(q^n).
	\end{eqnarray*}
	Since $\mathcal{D}(q^n)=-2\sum_{k=1}^n q^{n-k} \bar{q}^{k-1}$ is real, (see \cite{B}, \cite[Remark 4.4]{CDPS}), we get that
	\begin{equation}
		\label{f1}
		\mathcal{D}(q^{n+1})= (\mathcal{D}q^n)\bar{q}-2 q^n.
	\end{equation}
	By applying the conjugate Fueter operator $ \mathcal{\overline{D}}$ to formula \eqref{f1} and by the Leibnitz formula we get
	\begin{eqnarray*}
		\Delta (q^{n+1})&=& (\Delta q^n) \bar{q}+ (\mathcal{D} q^n) (\mathcal{\overline{D}} \bar{q})-2 \mathcal{\overline{D}} q^n\\
		&=& (\Delta q^n) \bar{q}-2 (\mathcal{D} q^n) -2 \mathcal{\overline{D}} q^n\\
		&=& (\Delta q^n) q_0- \mathcal{\overline{D}} q^n-(\Delta q^n) \underline{q}-2 \mathcal{D}q^n- \mathcal{\overline{D}} q^n.
	\end{eqnarray*}
	In order to prove the statement we have to show the following equality
	\begin{equation}
		\label{F5}
		(\Delta q^n) \underline{q}+2 \mathcal{D}q^n+\mathcal{\overline{D}} q^n=0.
	\end{equation}
By \cite[Thm. 3.2]{DKS} and Lemma \ref{res3b}  we have that
	\begin{eqnarray*}
		&& (\Delta q^n) \underline{q}+2 \mathcal{D}q^n+\mathcal{\overline{D}} q^n= -4 \sum_{k=1}^{n-1} (n-k) q^{n-k-1} \bar{q}^{k-1} \left(\frac{q- \bar{q}}{2}\right)\\
		&& \quad -4 \sum_{k=1}^n q^{n-k} \bar{q}^{k-1} +2 n q^{n-1}+2 \sum_{k=1}^n q^{n-k} \bar{q}^{k-1}\\
		&& = -2 \sum_{k=1}^n (n-k) q^{n-k} \bar{q}^{k-1}+2 \sum_{k=1}^{n-1} (n-k) q^{n-k-1} \bar{q}^k\\
		&& \quad - 2 \sum_{k=1}^n q^{n-k} \bar{q}^{k-1}+2n q^{n-1}\\
		&&=  -2 \sum_{k=1}^n (n-k) q^{n-k} \bar{q}^{k-1}+2 \sum_{k=0}^{n-1} (n-k) q^{n-k-1} \bar{q}^k\\
		&& \quad - 2 \sum_{k=1}^n q^{n-k} \bar{q}^{k-1}\\
		&&=0
	\end{eqnarray*}
\end{proof}

\begin{remark}
	If in formula \eqref{laplca1} we replace the operator $T$ with a generic quaternion $q \in \mathbb{H}$ and by considering $f(q)=q^n$ and $g(q)=q$, we get the same result of Lemma \ref{delta}. 
\end{remark}

By means of the resolvent equation \eqref{qres} it is also possible to study the Riesz projectors for the $Q$-functional calculus. 
\begin{theorem}
\label{rp}
Let $T=T_0e_0+T_1e_1+T_2e_2$ and assume that the operators $T_l$, $l=0,\, 1,\, 2$, have real spectrum. Let $\sigma_S(T)=\sigma_1\cup\sigma_2$ with $\operatorname{dist}(\sigma_1,\sigma_2)>0$.

Let $G_1,\, G_2\subset\mathbb H$ be two bounded slice Cauchy domains such that $\sigma_1\subset G_1$, $\overline G_1\subset G_2$ and $\operatorname{dist}(G_2,\sigma_2)>0$. Then the operator
$$
\tilde P:=\frac 1{2\pi}\int_{\partial (G_2\cap\mathbb C_J)} s\, ds_J\mathcal{Q}_{c,s}(T)^{-1}=\frac 1{2\pi}\int_{\partial (G_1\cap\cc_J)}\mathcal{Q}_{c,p}(T)^{-1} \, dp_J p
$$
is a projection, i.e.,
$$
\tilde P^2=\tilde P.
$$
\end{theorem}
\begin{proof}
From the definition of right $S$-resolvent operator we have
\begin{equation}
\label{auxh}
S^{-1}_R(s,T)= s\mathcal{Q}_{c,s}(T)^{-1}- \mathcal{Q}_{c,s}(T)^{-1}\bar{T}.
\end{equation}
By inserting formula \eqref{auxh} in the equation \eqref{qres} and by multiplying on the right by $p$ we get
\begin{eqnarray}
\label{auxh1}
&& \mathcal{Q}_{c,s}(T)^{-1} S^{-1}_L(p,T)p+s\mathcal{Q}_{c,s}(T)^{-1} \mathcal{Q}_{c,p}(T)^{-1}p- \mathcal{Q}_{c,s}(T)^{-1}\bar{T} \mathcal{Q}_{c,p}(T)^{-1}p\\
\nonumber
&&-2 \mathcal{Q}_{c,s}(T)^{-1} \underline{T} \mathcal{Q}_{c,p}(T)^{-1}p= [(\mathcal{Q}_{c,s}(T)^{-1}- \mathcal{Q}_{c,p}(T)^{-1})p- \bar{s}(\mathcal{Q}_{c,s}(T)^{-1}- \mathcal{Q}_{c,p}(T)^{-1})] \mathcal{Q}_s(p)^{-1}p.
\end{eqnarray}
Now, we multiply formula \eqref{auxh1} by $ds_J$ and we integrate it on $\partial(G_2 \cap \mathbb{C}_J)$ with respect to $ds_J$ and if we multiply on the right by $dp_J$ and we integrate on $ \partial(G \cap \mathbb{C}_J)$ with respect to $dp_J$. Therefore, we get
\begin{eqnarray*}
&& \! \! \! \! \! \! \! \! \! \! \int_{\partial(G_2 \cap \mathbb{C}_J)} ds_J\mathcal{Q}_{c,s}(T)^{-1} \int_{\partial(G_1 \cap \mathbb{C}_J)} S^{-1}_L(p,T)p+ \int_{\partial(G_2 \cap \mathbb{C}_J)} s ds_J\mathcal{Q}_{c,s}(T)^{-1} \int_{\partial(G_1 \cap \mathbb{C}_J)} \mathcal{Q}_{c,p}(T)^{-1}dp_Jp\\
&&\! \! \! \! \! \! \! \! \! \! -2 \int_{\partial(G_2 \cap \mathbb{C}_J)} ds_J \mathcal{Q}_{c,s}(T)^{-1} \underline{T} \int_{\partial(G_1 \cap \mathbb{C}_J)} \mathcal{Q}_{c,p}(T)^{-1}dp_Jp=  \\
&& \! \! \! \! \! \! \! \! \! \! \int_{\partial(G_2 \cap \mathbb{C}_J)} ds_J\int_{\partial(G_1 \cap \mathbb{C}_J)}[ (\mathcal{Q}_{c,s}(T)^{-1}-\mathcal{Q}_{c,p}(T)^{-1})p- \bar{s}(\mathcal{Q}_{c,s}(T)^{-1}-\mathcal{Q}_{c,p}(T)^{-1})] \mathcal{Q}_s(p)^{-1} dp_Jp.
\end{eqnarray*} 
By Lemma \ref{harmo} we get
$$
(2 \pi)^2 \tilde{P}^2=\int_{\partial(G_2 \cap \mathbb{C}_J)} ds_J\int_{\partial(G_1 \cap \mathbb{C}_J)}[ (\mathcal{Q}_{c,s}(T)^{-1}-\mathcal{Q}_{c,p}(T)^{-1})p- \bar{s}(\mathcal{Q}_{c,s}(T)^{-1}-\mathcal{Q}_{c,p}(T)^{-1})] \mathcal{Q}_s(p)^{-1} dp_Jp.
$$
Now, by \eqref{Cauchy9} we have
$$ \tilde{P}^2=\frac{1}{(2 \pi)^2}\int_{\partial(G_2 \cap \mathbb{C}_J)}\int_{\partial(G_1 \cap \mathbb{C}_J)} ds_J [\bar{s} \mathcal{Q}_{c,p}(T)^{-1}-\mathcal{Q}_{c,p}(T)^{-1}p] \mathcal{Q}_{s}(p)^{-1}  dp_Jp.$$
By exchanging the role of the integrals and Lemma \ref{app} with $B:=\mathcal{Q}_{c,p}(T)^{-1}$ we get
$$
\tilde{P}^2= \frac{1}{2 \pi} \int_{\partial(G_1 \cap \mathbb{C}_J)}  \mathcal{Q}_{c,p}(T)^{-1}dp_Jp= \tilde{P}.
$$
\end{proof}

We finish by making a table that sums up all the Riesz projectors in the $S$-functional, $Q$-functional, the polyanalytic functional  and $F$ functional calculi. We consider the sets $G_1$ as in the hypothesis of Theorem \ref{rp} and for the sake of simplicity we consider only the left case.  
\newline
\begin{center}
	\begin{tabular}{|c|c|c|}
		\hline
		\rule[-2mm]{-5mm}{-6cm}
		&$\displaystyle \hbox{Resolvent operator}$ & $\displaystyle \hbox{Riesz projectors}$  \\
		\hline
		$S$-functional &  $ \displaystyle S^{-1}_L(s,T)$& $\displaystyle \frac{1}{2 \pi}\int_{\partial(G_1 \cap \mathbb{C}_J)} S^{-1}_L(s,T) ds_J$\\
		\hline
		$Q$-functional &$ \displaystyle \mathcal{Q}_{c,s}(T)^{-1}$&$\displaystyle \frac{1}{2 \pi}\int_{\partial(G_1 \cap \mathbb{C}_J)}\mathcal{Q}_{c,s}(T)^{-1} ds_J s$\\
		\hline
		$P_2$-functional &$ \displaystyle \mathcal{P}_L^2(s,T)$& $\displaystyle\frac{1}{8 \pi}\int_{\partial(G_1 \cap \mathbb{C}_J)}\mathcal{P}_L^2(s,T) ds_J s$\\
		\hline
		$F$-functional& $\displaystyle F_L(s,T)$&$\displaystyle- \frac{1}{8 \pi}\int_{\partial(G_1 \cap \mathbb{C}_J)}F_L(s,T) ds_J s^2 $\\
		\hline
	\end{tabular}
\end{center}

From this table it is clear that to have a suitable definition of Riesz projector in the functional calculi based on the $S$-spectrum we have to suitably integrate the respective resolvent operator multiplied by a monomial of a certain degree.

\section{Concluding remarks}
In this paper we show that by applying the conjugate Fueter operator, $\mathcal{\overline{D}}$, to a slice hyperholomorphic function  we get, as a consequence of the Fueter theorem, a polyanalytic function of order 2. To this set of functions we associate a functional calculus. If we consider the Fueter theorem in Clifford algebras in dimension at least five it is possible to study polyanalytic functional calculi of order higher than two.

\hspace{4mm}

\noindent
Antonino De Martino,
Dipartimento di Matematica \\ Politecnico di Milano\\
Via Bonardi n.~9\\
20133 Milano\\
Italy

\noindent
\emph{email address}: antonino.demartino@polimi.it\\

\vspace*{5mm}
\noindent
Stefano Pinton,
Dipartimento di Matematica \\ Politecnico di Milano\\
Via Bonardi n.~9\\
20133 Milano\\
Italy

\noindent
\emph{email address}: stefano.pinton@polimi.it\\

\end{document}